\documentclass[11pt]{amsart}
\usepackage{amsfonts,amscd}
\usepackage{graphicx}
\usepackage{amsmath,latexsym,amssymb,amsthm}
\usepackage{hyperref}
\usepackage{cite}
\usepackage{lscape,fancyhdr}
\usepackage{a4}

\newcounter{cnt1}
\newcounter{cnt2}
\newcounter{cnt3}
\newcommand{\blr}{\begin{list}{$($\roman{cnt1}$)$} {\usecounter{cnt1}
        \setlength{\topsep}{0pt} \setlength{\itemsep}{0pt}}}
\newcommand{\bla}{\begin{list}{$($\alph{cnt2}$)$} {\usecounter{cnt2}
        \setlength{\topsep}{0pt} \setlength{\itemsep}{0pt}}}
\newcommand{\bln}{\begin{list}{$($\arabic{cnt3}$)$} {\usecounter{cnt3}
                \setlength{\topsep}{0pt} \setlength{\itemsep}{0pt}}}
\newcommand{\el}{\end{list}}

\newtheorem{Thm}{Theorem}[section]
\newtheorem{Lem}[Thm]{Lemma}
\newtheorem{Prop}[Thm]{Proposition}
\newtheorem{Def}[Thm]{Definition}
\newtheorem{Exm}[Thm]{Example}
\newtheorem{Rem}[Thm]{Remark}
\newtheorem{Cor}[Thm]{Corollary}

\begin{document}

\title{On Hom-Gerstenhaber algebras and Hom-Lie algebroids}
\author{Ashis Mandal and Satyendra Kumar Mishra}
\footnote{The research of S.K. Mishra is supported by CSIR-SPM fellowship grant 2013 .}

\begin{abstract}

We define the notion of hom-Batalin-Vilkovisky algebras and strong differential hom-Gerstenhaber algebras as a special class of hom-Gerstenhaber algebras and provide canonical examples associated to some well-known hom-structures. Representations of a hom-Lie algebroid on a hom-bundle are defined and a cohomology of a regular hom-Lie algebroid with coefficients in a representation is studied. We discuss about relationship between these classes of hom-Gerstenhaber algebras and geometric structures on a vector bundle. As an application, we associate a homology to a regular hom-Lie algebroid and then define a hom-Poisson homology associated to a hom-Poisson manifold.   
\end{abstract}
\maketitle

\section{Introduction}
Hom-Lie algebras were introduced in the context of $q$-deformation of Witt and Virasoro algebras. In a sequel, various concepts and properties have been derived to the framework of other hom-algebras as well. The study of hom-algebras appeared extensively in the work of J. Hartwig, D. Larsson, A. Makhlouf, S. Silvestrov, D. Yau and other authors (\cite{DefHLIE}, \cite{HLIE01}, \cite{HALG}, \cite{NtHOM},\cite{HALG2}). 

More recently, the notion of hom-Lie algebroid is introduced in \cite{hom-Lie} by going through a formulation of hom-Gerstenhaber algebra and following the classical bijective correspondence between Lie algebroids and Gerstenhaber algebras. On the other hand, there are canonically defined adjoint functors between category of Lie-Rinehart algebras and category of Gerstenhaber algebras. This leads us to proceed further and define the category of hom-Lie-Rinehart algebras in \cite{old}, and discuss adjoint functors between the category of hom-Gerstenhaber algebras and the category of hom-Lie-Rinehart algebras. Furthermore, the notion of a hom-Lie bialgebroid and a hom-Courant algebroid is defined in \cite{hom-Lie1}.

There are some well known algebraic structures such as Batalin-Vilkovisky algebras and differential Gerstenhaber algebras satisfying nice relationship with Lie-Rinehart algebras and different geometric stuctures on a vector bundle (in \cite{Xu}, \cite{YK1}, \cite{Hueb3} and the references therein). In this paper, our first goal is to define hom-Batalin-Vilkovisky algebras and strong differential hom-Gerstenhaber algebras as a special class of hom-Gerstenhaber algebras and present several examples which are obtained canonically. 
 
 In \cite{hom-Lie1}, a hom-bundle of rank $n$ is defined, as a triplet $(A,\psi,\alpha)$ where $A$ is a rank $n$ vector bundle over a smooth manifold  $M$, the map $\psi:M\rightarrow M$ is a smooth map and $\alpha:\Gamma A\rightarrow \Gamma A$ is a $\psi^*$-function linear map, i.e. $\alpha(f.x)=\psi^*(f).\alpha(x)$ for $f\in C^{\infty}(M), ~x\in \Gamma A$. It is proved in \cite{hom-Lie} that hom-Lie algebroid structures on the hom-bundle $(A,\psi,\alpha)$ are in bijective correspondence with hom-Gerstenhaber algebra structures on $(\mathfrak{A},\tilde{\alpha})$, where $\mathfrak{A}=\oplus_{0\leq k\leq n}\Gamma (\wedge^kA)$ and the map $\tilde{\alpha}$ is an extension of $\alpha$ to higher degree elements. Let $(A,\psi,\alpha)$ be an invertible hom-bundle. Then, the second goal of this paper is to analyse the following.
\begin{itemize}
\item The relationship between hom-Batalin-Vilkovisky algebra structures on $(\mathfrak{A},\tilde{\alpha})$ and hom-Lie algebroid structures on the hom-bundle $(A,\psi,\alpha)$ with a representation of this hom-Lie algebroid on the hom-bundle $(\wedge^n A,\psi,\tilde{\alpha})$.
\item The relationship between strong differential hom-Gerstenhaber algebra structures on $(\mathfrak{A},\tilde{\alpha})$ and hom-Lie bialgebroid structures on the hom-bundle $(A,\psi,\alpha)$.
\end{itemize}
For the first relationship, we prove a more general result (Corollary \ref{Res1}) for hom-Lie-Rinehart algebras (see Definition \ref{hom-LR}). As an application, we get a homology associated to a hom-Poisson manifold $(M,\psi,\pi)$ where the map $\psi:M\rightarrow M$ is a diffeomorphism. We call this homology as ``hom-Poisson homology" since for the case when $\psi$ is the identity map on $M$, this gives the Poisson homology. The paper is organized as follows:

In Section $2$, we recall preliminaries on hom-structures. This will help us to fix several notations to be used in the later part of our discussion.

 In Section $3$, we define a hom-Batalin-Vilkovisky algebra as an exact hom-Gerstenhaber algebra and give some examples. Right modules and subsequently a homology is studied for a hom-Lie-Rinehart algebra with coefficients in a right module. We also express  homology of a hom-Lie-Rinehart algebra with trivial coefficients in terms of the associated hom-Gerstenhaber algebra. We study a cohomology of regular hom-Lie-Rinehart algebra $(\mathcal{L},\alpha)$ with coefficients in a left module $(M,\beta)$. If the underlying $A$-module $L$ is projective of rank $n$, then we prove a one-one correspondence between right $(\mathcal{L},\alpha)$-module structures on $(A,\phi)$ and left $(\mathcal{L},\alpha)$-module structures on $(\wedge^n L,\alpha)$.

In Section $4$, we define representations of a hom-Lie algebroid on a hom-bundle. If $\mathcal{A}:=(A,\phi,[-,-],\rho,\alpha)$ is a regular hom-Lie algebroid, where $A$ is a rank $n$ vector bundle over $M$, then it is proved that there exists a bijective correspondence between exact generators of the associated hom-Gerstenhaber algebra and representations of $\mathcal{A}$ on the hom-bundle  $(\wedge^n A,\psi,\tilde{\alpha})$. We associate a homology to a regular hom-Lie algebroid when the underlying vector bundle is a real vector bundle. For a hom-Poisson manifold $(M,\psi,\pi)$ (where $\psi:M\rightarrow M$ is a diffeomorphism), we define hom-Poisson homology as the homology of its associated cotangent hom-Lie algebroid. Moreover, we define a cochain complex for a regular hom-Lie algebroid with coefficients in a representation.   

In the last Section, we define strong differential hom-Gerstenhaber algebras and give several examples. Finally, We discuss the relationship between strong differential hom-Gerstenhaber algebras and hom-Lie bialgebroids.
 
\section{Preliminaries on hom-structures}
In this section, we recall some basic definitions from \cite{hom-Lie}, \cite{old}, and \cite{OnhLie}. Let $R$ be a commutative ring with unity and $\mathbb{Z}_+$ be the set of all non-negative integers. 
\begin{Def}
A hom-Lie algebra is a triplet $(\mathfrak{g},[-,-],\alpha)$ where $\mathfrak{g}$ is a $R$-module equipped with a skew-symmetric $R$-bilinear map $[-,-]:\mathfrak{g}\times \mathfrak{g}\rightarrow \mathfrak{g}$ and a  Lie algebra homomorphism $\alpha:\mathfrak{g}\rightarrow \mathfrak{g}$ such that the hom-Jacobi identity holds, i.e.,
$ [\alpha(x),[y,z]]+[\alpha(y),[z,x]]+[\alpha(z),[x,y]] = 0$
for all $x,y,z\in \mathfrak{g}$.

If $\alpha$ is an automorphism of $\mathfrak{g}$, then $(\mathfrak{g},[-,-],\alpha)$ is called a regular hom-Lie algebra.
\end{Def}

\begin{Def}\label{Rep-hom-Lie}
A representation of a hom-Lie algebra $(\mathfrak{g},[-,-],\alpha)$ on a $R$-module $V$ is a pair $(\rho, \alpha_V)$ of $R$-linear maps $\rho: \mathfrak{g}\rightarrow \mathfrak{g}l(V),~\alpha_V: V\rightarrow V$ such that 
\begin{equation*}
\rho(\alpha(x))\circ \alpha_V= \alpha_V\circ \rho(x),
\end{equation*} and
\begin{equation*}
\rho([x,y])\circ \alpha_V= \rho(\alpha(x))\circ\rho(y)-\rho(\alpha(y))\circ\rho(x).
\end{equation*}
for all $x,y \in \mathfrak{g}$.
\end{Def}
\begin{Exm}
For any integer s, we can define the $\alpha^s$-adjoint representation of the regular hom-Lie algebra $(\mathfrak{g},[-,-],\alpha)$ on $\mathfrak{g}$ by $(ad_s,\alpha)$, where
$ad_s(g)(h)=[\alpha^s(g),h]~~~\mbox{for all}~g,h\in \mathfrak{g}.$
\end{Exm}

\begin{Def}
A graded hom-Lie algebra is a triplet $(\mathfrak{g},[-,-],\alpha)$, where $\mathfrak{g}=\oplus_{i\in \mathbb{Z}_+}\mathfrak{g}_i$ is a graded module, $[-,-]:\mathfrak{g}\otimes\mathfrak{g}\rightarrow \mathfrak{g}$ is a graded skew-symmetric bilinear map of degree $-1$, and $\alpha:\mathfrak{g}\rightarrow \mathfrak{g}$ is a homomorphism of $(\mathfrak{g},[-,-])$ of degree $0$, satisfying:
$$(-1)^{(i-1)(k-1)}[\alpha(x),[y,z]]+(-1)^{(j-1)(i-1)}[\alpha(y),[z,x]]+(-1)^{(k-1)(i-1)}[\alpha(z),[x,y]]=0,$$ 
 for all $x\in \mathfrak{g}_i,~y\in \mathfrak{g}_j,$ and $z\in \mathfrak{g}_k$.
\end{Def}
\begin{Def}
Let $\mathfrak{A}=\oplus_{k\in\mathbb{Z}_+}\mathfrak{A}_k$ be a graded commutative algebra, $\sigma$ and $\tau$ be $0$-degree endomorphism of $\mathfrak{A}$, then a $(\sigma,\tau)$-differential graded commutative algebra is quadruple $(\mathfrak{A},\sigma,\tau,d)$, where $d$ is a degree $1$ square zero operator on $\mathfrak{A}$ satisfying the following:
\begin{enumerate}
\item $d\circ\sigma=\sigma\circ d,~ ~d\circ\tau=\tau\circ d;$
\item $d(ab)=d(a)\tau(b)+(-1)^{|a|}\sigma(a)d(b)$ for $a,b\in \mathfrak{A}$.
\end{enumerate}
 \end{Def}
\begin{Def}
A Gerstenhaber algebra is a triple  $(\mathcal{A}=\oplus_{i\in\mathbb{Z}_+}\mathcal{A}_i,\wedge,[-,-])$ where $\mathcal{A}$ is a graded commutative associative $R$-algebra, and $[-,-]:\mathcal{A}\otimes \mathcal{A}\rightarrow \mathcal{A}$ is a bilinear map of degree $-1$ such that:
\begin{enumerate}
\item $(\mathcal{A},[-,-])$ is a graded Lie algebra.
\item The following Leibniz rule holds:
      $$[X,Y\wedge Z]=[X,Y]\wedge Z+(-1)^{(i-1)j}Y\wedge [X,Z], $$
      for all $X\in \mathcal{A}_i,~Y\in \mathcal{A}_j,~Z\in \mathcal{A}_k$.
\end{enumerate}
\end{Def}
For a Gerstenhaber algebra $(\mathcal{A}=\oplus_{i\in\mathbb{Z}_+}\mathcal{A}_i,\wedge,[-,-])$, a linear operator  $D:\mathcal{A}\rightarrow \mathcal{A}$ of degree $-1$ is said to be a generator of the bracket $[-,-]$ if for every homogeneous $a,b \in \mathcal{A}$:
$$[a,b]=(-1)^{|a|}(D(ab)-(Da)b-(-1)^{|a|}a(Db));$$ 
\begin{Def}
A Gerstenhaber algebra $\mathcal{A}$ is called an \textbf{exact Gerstenhaber algebra} or a \textbf{Batalin-Vilkovisky algebra} if it has a generator $D$ of square zero. 
\end{Def}
\begin{Def}\label{DerS}
 Given an associative commutative algebra $A$, an $A$-module $M$ and an algebra endomorphism $\phi: A\rightarrow A$, we call a $R$-linear map $\delta: A\rightarrow M$ a $\phi$-derivation of $A$ into $M$ if it satisfies 
$$\delta(a.b)=\phi(a).\delta(b)+\phi(b).\delta(a)$$
for all $a,b \in A$. Let us denote the set of $\phi$-derivations by $Der_{\phi}(A)$.
\end{Def} 
\begin{Rem}
Given a manifold $M$ and a smooth map $\psi: M\rightarrow M $, then the space of $\psi^*$-derivations of $C^{\infty}(M)$ into $C^{\infty}(M)$ can be identified with the space of sections of the pull back bundle of TM, i.e. $\Gamma(\psi^!TM).$ 
\end{Rem}
\begin{Def}\label{hom-Lie}
 A hom-Lie algebroid is a quintuple $(A,\phi,[-,-],\rho,\alpha)$, where $A$ is a vector bundle over a smooth manifold $M$, $\phi:M\rightarrow M$ is a smooth map, $[-,-]:\Gamma(A)\times \Gamma(A)\rightarrow \Gamma(A)$ is a bilinear map, $\rho: \phi^{!}A\rightarrow \phi^!TM$ is a bundle map, called anchor, and $\alpha: \Gamma (A)\rightarrow\Gamma(A)$ is a linear map, such that:
\begin{enumerate}
\item $\alpha(f.X)=\phi^*(f).\alpha(X)$ for all $X\in \Gamma(A),~f\in C^{\infty}(M)$;
\item The triplet $(\Gamma(A),[-,-],\alpha)$ is a hom-Lie algebra;
\item The following hom-Leibniz identity holds:
$[X,fY]=\phi^*(f)[X,Y]+\rho(X)[f]\alpha(Y);$
for all $X,Y\in \Gamma(A),~f\in C^{\infty}(M)$;
\item The pair $(\rho,\phi^*)$ is a representation of the hom-Lie algebra $(\Gamma(A),[-,-],\alpha)$ on the space $C^{\infty}(M).$
\end{enumerate}
Here, $\rho(X)[f]$ stands for the function on $M$, such that $$\rho(X)[f](m)=\big<d_{\phi(m)}f,\rho_m(X_{\phi(m)}) \big>$$for $m\in M$. Here, $\rho_m:(\phi^!A)_m\cong A_{\phi(m)}\rightarrow (\phi^!TM)_m\cong T_{\phi(m)}M$ is the anchor map evaluated at $m\in M$ and $X_{\phi(m)}$ is the value of the section $X\in\Gamma (A)$ at $\phi(m)\in M$.

Moreover, a hom-Lie algebroid is said to be regular (or invertible) if the map $\alpha: \Gamma (A)\rightarrow\Gamma(A)$ is an invertible linear map and the map $\phi:M\rightarrow M $ is a diffeomorphism. 
\end{Def}

\begin{Def}\label{HGR}
A hom-Gerstenhaber algebra is a quadruple $(\mathcal{A}=\oplus_{i\in\mathbb{Z}_+}\mathcal{A}_i,\wedge,[-,-],\alpha)$ where $\mathcal{A}$ is a graded commutative associative $R$-algebra, $\alpha$ is an endomorphism of $(\mathcal{A},\wedge)$ of degree $0$ and $[-,-]:\mathcal{A}\otimes \mathcal{A}\rightarrow \mathcal{A}$ is a bilinear map of degree $-1$ such that:
\begin{enumerate}
\item the triple $(\mathcal{A},[-,-],\alpha)$ is a graded hom-Lie algebra.
\item the hom-Leibniz rule holds:
      $$[X,Y\wedge Z]=[X,Y]\wedge \alpha(Z)+(-1)^{(i-1)j}\alpha(Y)\wedge [X,Z], $$
      for all $X\in \mathcal{A}_i,~Y\in \mathcal{A}_j,~Z\in \mathcal{A}_k$.
\end{enumerate}
\end{Def}
\begin{Exm}\label{Ger}
Let $(\mathcal{A},[-,-],\wedge)$ be a Gerstenhaber algebra and $\alpha:(\mathcal{A},[-,-],\wedge)\rightarrow (\mathcal{A},[-,-],\wedge)$ be an endomorphism, then the quadruple $(\mathcal{A},\wedge,\alpha\circ [-,-],\alpha)$ is a hom-Gerstenhaber algebra.  
\end{Exm}

\begin{Exm}\label{Hom-Ger2}
Given a hom-Lie algebra $(\mathfrak{g},[-,-],\alpha)$, one can define a hom-Gerstenhaber algebra $(\mathfrak{G}=\wedge^*\mathfrak{g},\wedge,[-,-]_{\mathfrak{G}},\alpha_{\mathfrak{G}})$, where 
$$[x_1\wedge\cdots\wedge x_n,y_1\wedge \cdots\wedge y_m]_{\mathfrak{G}}=\sum_{i=1}^n\sum_{j=1}^m (-1)^{i+j}[x_i,y_j]\wedge \alpha_G(x_1\wedge\cdots \hat{x_i}\wedge\cdots\wedge x_n\wedge y_1\wedge \cdots \hat{y_j}\wedge\cdots\wedge y_m)$$
for all $x_1,\cdots,x_n,y_1,\cdots,y_m\in \mathfrak{g}$ and 
$$\alpha_{\mathfrak{G}}(x_1\wedge\cdots\wedge x_n)=\alpha(x_1)\wedge\cdots \wedge \alpha(x_n).$$
See \cite{hom-Lie} for further details.
\end{Exm}
\begin{Def}\label{hom-LR}
A hom-Lie Rinehart algebra over $(A,\phi)$  is a tuple $(A,L,[-,-],\phi,\alpha,\rho)$ ,
 where $A$ is an associative commutative algebra, $L$ is an $A$-module, $[-,-]:L\times L\rightarrow L$ is a skew symmetric bilinear map, $\phi:A\rightarrow A$ is an algebra homomorphism, $\alpha:L\rightarrow L $ is a linear map satisfying $\alpha([x,y])=[\alpha(x),\alpha(y)]$, and $\rho: L\rightarrow Der_{\phi}A$ such that 
\begin{enumerate}
\item The triple $(L,[-,-],\alpha)$ is a hom-Lie algebra.
\item $\alpha(a.x)=\phi(a).\alpha(x)$ for all $a\in A,~x\in L $.
\item $(\rho, \phi)$ is a representation of $(L,[~,~],\alpha)$ on $A$.
\item $\rho(a.x)=\phi(a).\rho(x)$ for all $a\in A,~x\in L $.
\item $[x,a.y]=\phi(a)[x,y]+ \rho(x)(a)\alpha(y)$ for all $a\in A,~x,y\in L $.
\end{enumerate}

A hom-Lie-Rinehart algebra $(A,L,[-,-],\phi,\alpha,\rho)$ is called regular if the map $\phi:A\rightarrow A$ is an algebra automorphism and the map $\alpha:L\rightarrow L $ is a bijection. 
\end{Def}
Let us denote a hom-Lie-Rinehart algebra $(A,L,[-,-],\phi,\alpha,\rho)$ simply by $(\mathcal{L},\alpha)$. In \cite{old}, we discussed several examples of hom-Lie-Rinehart algebras.

\section{Hom-Batalin-Vilkovisky algebras and hom-Lie-Rinehart algebras}

\subsection{\textbf{Hom-Batalin-Vilkovisky algebras}}
\begin{Def}
A hom-Gerstenhaber algebra $(\mathcal{A}=\oplus_{i\in\mathbb{Z}_+}\mathcal{A}_i,\wedge,[-,-],\alpha)$ is said to be generated by an operator $D:\mathcal{A}\rightarrow \mathcal{A}$ of degree $-1$ if $D\circ \alpha=\alpha\circ D$, and 

\begin{equation}\label{generator}
[X,Y]=(-1)^{|X|}(D(XY)-(DX)\alpha(Y)-(-1)^{|X|}\alpha(X)(DY));
\end{equation}
for homogeneous elements $X,Y\in \mathcal{A}$. Furthermore, if $D^2=0$, then $D$ is called an exact generator and hom-Gerstenhaber algebra $\mathcal{A}$ is called an exact hom-Gerstenhaber algebra or hom-Batalin-Vilkovisky algebra.  

\end{Def}  

\begin{Exm}\label{BV1}
Recall from Example \ref{Hom-Ger2} that there is a canonical hom-Gerstenhaber algebra structure on the exterior algebra $(\mathfrak{G}=\wedge^*\mathfrak{g},\wedge,[-,-]_{\mathfrak{G}},\alpha_{\mathfrak{G}})$ associated to a hom-Lie algebra $(\mathfrak{g},[-,-],\alpha)$. If we consider the boundary operator $d:\wedge^n\mathfrak{g}\rightarrow \wedge^{n-1}\mathfrak{g}$ of a hom-Lie algebra with coefficients in the trivial module $R$ (defined in \cite{HALG2}), given by 
$$d(x_1\wedge\cdots\wedge x_n)=\sum_{1\leq i<j\leq n} (-1)^{i+j}[x_i,x_j]\wedge \alpha_G(x_1\wedge\cdots \hat{x_i}\wedge\cdots\wedge \hat{x_j}\wedge\cdots\wedge x_n)$$
for all $x_1,\cdots,x_n\in \mathfrak{g}$. Then, $d:\mathfrak{G}\rightarrow \mathfrak{G}$ is a map of degree $-1$ such that $d^2=0$, and $d\circ \alpha=\alpha\circ d$. More importantly, this operator $d$ generates the graded hom-Lie bracket $[-,-]_{\mathfrak{G}}$ in the following way:

\begin{equation}
[X,Y]_{\mathfrak{G}}=(-1)^{|X|}(d(XY)-(dX)\alpha_{\mathfrak{G}}(Y)-(-1)^{|X|}\alpha_{\mathfrak{G}}(X)(dY))
\end{equation}  
for $X,Y\in \mathfrak{G}$. Hence the operator $d$ is an exact generator of the hom-Gerstenhaber algebra $\mathfrak{G}$. This yields a \textbf{\it hom-BV algebra} associated to a hom-Lie algebra.
\end{Exm}

\begin{Exm}\label{Ger2}
Given a Batalin-Vilkovisky algebra $(\mathcal{A},[-,-],\wedge,\partial)$ and an endomorphism $\alpha:(\mathcal{A},[-,-],\wedge)\rightarrow (\mathcal{A},[-,-],\wedge)$ of the underlying Gerstenhaber algebra satisfying $\partial\circ \alpha=\alpha\circ \partial$, then the quadruple $(\mathcal{A},\wedge,[-,-]_{\alpha}=\alpha\circ [-,-],\alpha)$ is a hom-Gerstenhaber algebra. Since for any homogeneous elements $a,b \in \mathcal{A}$, we have
$$[a,b]=(-1)^{|a|}(\partial(ab)-\partial(a)b-(-1)^{|a|}a\partial(b)),$$
so by applying $\alpha$ on both sides we obtain the following relation:
$$[a,b]_{\alpha}=(-1)^{|a|}(\partial_{\alpha}(ab)-\partial_{\alpha}(a)\alpha(b)-(-1)^{|a|}\alpha(a)\partial_{\alpha}(b)),$$
i.e. $\partial_{\alpha}=\alpha\circ\partial$ generates the hom-Gerstenhaber bracket $[-,-]_{\alpha}$.  
\end{Exm}

\subsection{\textbf{Homology of a hom-Lie-Rinehart algebra}}

Let $A$ be an associative and commutative $R$-algebra, and $\phi$ be an algebra homomorphism of $A$. Suppose $ (\mathcal{L},\alpha) $ is a hom-Lie-Rinehart algebra over the pair $(A, \phi)$.

\begin{Def}
Let $M$ be an $A$-module and $\beta\in End_{R}(M)$. Then the pair $(M,\beta)$ is a right module over a hom-Lie Rinehart algebra $(\mathcal{L},\alpha)$ if the following conditions hold.
\begin{enumerate}
\item There is a map $\theta:M\otimes L\rightarrow M$ such that the pair  $(\theta,\beta)$ is a representation of the hom-Lie algebra $(L,[-,-],\alpha)$ on $M$, where  $\theta(m,x)$ is usually denoted by $\{m,x\}$ for $x\in L,~m\in M$.
\item $\beta(a.m)=\phi(a).\beta(m)$ for $a \in A$ and $m \in M$.
\item $\{a.m,x\}=\{m,a.x\}=\phi(a).\{m,x\}-\rho(x)(a).\beta(m)$ for $a\in A,~x\in L,~m\in M$.
\end{enumerate} 
\end{Def}
If $\alpha=Id_L$ and $\beta =Id_M$, then M is a right Lie-Rinehart algebra module. Note that there is no canonical right module structure on $(A,\phi)$ as one would expect from the case of Lie-Rinehart algebras.

Let $(M,\beta)$ be a right module over a hom-Lie-Rinehart algebra $(\mathcal{L},\alpha)$. Define $C_n(\mathcal{L},M):=M\otimes_A \wedge_A^n L,$
for $n\geq 0$. Also define the boundary map $d:C_n(\mathcal{L},M)\rightarrow C_{n-1}(\mathcal{L},M)$ by
\begin{equation}\label{bound}
\begin{split}
d(m\otimes (x_1\otimes\cdots\otimes x_n))&=\sum_{i=1}^n (-1)^{i+1}\{m,x_i\}\otimes(\alpha(x_1)\otimes\cdots\otimes\hat{\alpha(x_i)}\otimes\cdots\otimes \alpha(x_n))+\\
&\sum_{i<j}(-1)^{i+j} \beta(m)\otimes([x_i,x_j],\alpha(x_1)\otimes\cdots\otimes\hat{\alpha(x_i)}\otimes\cdots\otimes\hat{\alpha(x_j)}\otimes\cdots\otimes \alpha(x_n))
\end{split}
\end{equation}
for $m \in M$ and $x_1,\cdots,x_n \in L$.
By the definition of the right hom-Lie-Rinehart algebra module structure on $(M,\beta)$, it follows that $d^2=0$. Thus, $(C_*(\mathcal{L},M),d)$ forms a chain complex. The homology of the hom-Lie-Rinehart algebra $(\mathcal{L},\alpha)$ with coefficient in the right module $(M,\beta)$ is given  by
$$H_*^{hLR}(L,M):= H_*(C_*(\mathcal{L},M)).$$ 

\begin{Rem}
If $\alpha=Id_L$ and $\beta=Id_M$, then $M$ is a right Lie-Rinehart algebra module and $H_*^{hLR}(L,M)$ is the Lie-Rinehart algebra homology with coefficients in $M$. 
\end{Rem}

\begin{Rem}
If $A=R$ then $(\theta,\beta)$ is a representation of hom-Lie algebra $(L,[-,-],\alpha)$ on $M$ and $H_*^{hLR}(L,M)$ is the homology of a hom-Lie algebra with coefficients in $M$, which is defined in \cite{HALG2}. 
\end{Rem}

Let $(\mathcal{L},\alpha)$ be a hom-Lie-Rinehart algebra, then there is no canonical right module structure on $(A,\phi)$. In the following theorem we give a bijective correspondence between the exact generators of hom-Gerstenhaber bracket on $\wedge^*_AL$ and right module structures on $(A,\phi)$.

\begin{Thm}\label{Corres1}
There is a bijective correspondence between right $(\mathcal{L},\alpha)$-module structures on $(A,\phi)$ and exact generators of the hom-Gerstenhaber algebra bracket on $\wedge^*_A L$.

\begin{proof}
Let $(A,\phi)$ be a right module structure over hom-Lie-Rinehart algebra $(\mathcal{L},\alpha)$, then we can define an operator $D:\wedge^*_AL\rightarrow \wedge^*_AL$ as follows:
\begin{equation}
\begin{split}
D(x_1\wedge\cdots\wedge x_n)=&\sum_{i=1}^n (-1)^{i+1}\{1,x_i\}.(\alpha(x_1)\wedge\cdots\wedge\hat{\alpha(x_i)}\wedge\cdots\wedge \alpha(x_n))\\
&+\sum_{i<j}(-1)^{i+j} ([x_i,x_j]\wedge\alpha(x_1)\wedge\cdots\wedge\hat{\alpha(x_i)}\wedge\cdots\wedge\hat{\alpha(x_j)}\wedge\cdots\wedge \alpha(x_n))
\end{split}
\end{equation}  

where $\{-,-\}$ denotes the right action of $L$ on $A$. It follows from straightforward calculations that $D$ commutes with $\alpha$ and generates the bracket on $\wedge_A^*L$, i.e. 
$$[X,Y]=(-1)^{|X|}(D(XY)-(DX)\alpha(Y)-(-1)^{|X|}\alpha(X)(DY));$$
for any homogeneous elements $X,Y\in \wedge_A^*L$.

Conversely, Let $D:\wedge^*_AL\rightarrow \wedge^*_AL$ generates the hom-Gerstenhaber bracket on $\wedge^*_AL$ and $D^2=0$. Then define the right action of $L$ on $A$ as follows
\begin{align*}
\{ab,x\}&=\phi(ab)(D(x))-\rho(x)(ab)\\
&=\phi(a)(\phi(b)(D(x))-\rho(x)(b))-\phi(b)\rho(x)(a)\\
& =\phi(a)(\{b,x\})-\phi(b)\rho(x)(a).
\end{align*}
Here, $\rho$ is the anchor map of the hom-Lie-Rinehart algebra $(\mathcal{L},\alpha)$. Furthermore, using equation \eqref{generator} we have the following:
\begin{align*}
\{b,ax\}&=\phi(b)(D(ax))-\rho(ax)(b)\\
&=\phi(b)(\phi(a)(D(x))-\rho(x)(a))-\phi(a)\rho(x)(b)\\
& =\phi(a)(\{b,x\})-\phi(b)\rho(x)(a).
\end{align*}
  
  Let us define $\theta:A\otimes L\rightarrow A$ as $\theta(a,x)=\{a,x\}$, then the conditions: $D^2=0$, and $D\circ \alpha=\alpha\circ D$ implies that $(\theta,\phi)$ is a representation of the hom-Lie algebra $(L,[-,-],\alpha)$ on $A$. Hence, an exact generator gives a right $(\mathcal{L},\alpha)$-module structure on the pair $(A,\phi)$.   

\end{proof}
\end{Thm}
Suppose the canonical hom-Gerstenhaber algebra structure on $\wedge^*_A L$ corresponding to the hom-Lie-Rinehart algebra $(\mathcal{L},\alpha)$, has an exact generator $D$. Then by the above Theorem \ref{Corres1}, we get the corresponding right $(\mathcal{L},\alpha)$-module structure on $(A,\phi)$. Moreover, by the canonical isomorphism of $A$-modules: $A\otimes_A \wedge_A^n L\cong \wedge_A^n L$, and the definition of the boundary from equation (\ref{bound}), we have the following result:

\begin{Prop}
The homology $H_*^{hLR}(\mathcal{L},A)$ of hom-Lie-Rinehart algebra $(\mathcal{L},\alpha)$ with coefficients in right $(\mathcal{L},\alpha)$-module $(A,\phi)$ is isomorphic to the homology of the chain complex $(\wedge^*_A L,D)$.
\end{Prop}

\begin{Def}
Let $M$ be an $A$-module, and $\beta\in End_{R}(M)$. Then the pair $(M,\beta)$ is a left-module over the hom-Lie Rinehart algebra $(\mathcal{L},\alpha)$ if following conditions hold.
\begin{enumerate}
\item There is a map  $\theta:L\otimes M\rightarrow M$, such that the pair
$(\theta,\beta)$ is a representation of the hom-Lie algebra  $(L,[-,-],\alpha)$ on $M$. 
 We denote $\theta(X)(m)$ by $\{X,m\}$ for $X \in L$ and $m \in M$.
\item $\beta(a.m)=\phi(a).\beta(m)$ for $a \in A $ and $m \in M$.
\item $\{a.X, m\}=\phi(a)\{X,m\}$ for  $a\in A,~X\in L,~m\in M$.
\item $\{X, a.m\}=\phi(a)\{X,m\}+\rho(X)(a).\beta(m)$ for $X\in L,~a\in A,~m\in M$.
\end{enumerate} 
\end{Def}

If we consider $\alpha=Id_L$, and $\beta=Id_M$, then the hom-Lie-Rinehart algebra $(\mathcal{L},\alpha)$ turns out to be simply a Lie-Rinehart algebra and any left $(\mathcal{L},\alpha)$-module $(M,\beta)$ is just a Lie-Rinehart algebra module.

\begin{Exm}
Suppose $ (\mathcal{L},\alpha) $ is a hom-Lie-Rinehart algebra over $(A,\phi)$. Then $(A,\phi)$ is a left $(\mathcal{L},\alpha)$-module, where left action is given by the anchor map. 
\end{Exm}

\subsection{\textbf{Regular Hom-Lie-Rinehart algebras}}
Let $(\mathcal{L},\alpha)$ be a regular hom-Lie-Rinehart algebra over $(A,\phi)$ and $(M,\beta)$ be a left $(\mathcal{L},\alpha)$-module. Let $\theta:L\times M\rightarrow M$ be the left action of $L$ on $M$, then define a cochain complex $$(Alt_A(\mathcal{L},M)=\oplus_{n\geq 0}Alt_A^n(\mathcal{L},M),\delta),$$ 
where $Alt_A^n(\mathcal{L},M):= Hom_A(\wedge_A^n L,M)$, and the coboundary map $\delta:Alt_A^n(\mathcal{L},M)\rightarrow Alt_A^{n+1}(\mathcal{L},M)$ is defined as follows:
\begin{equation}
\begin{split}
\delta f(x_1,\cdots,x_{n+1})= &\sum_{i=1}^{n+1}(-1)^{i+1}\theta(\alpha^{-1}(x_i))(f(\alpha^{-1}(x_1),\cdots,\hat{x_i},\cdots,\alpha^{-1}(x_{n+1}))\\&+\sum_{1\leq i<j\leq n+1}(-1)^{i+j}\beta(f(\alpha^{-2}([x_i,x_j]),\alpha^{-1}(x_1),\cdots,\hat{x_i},\cdots,\hat{x_j},\cdots,\alpha^{-1}(x_{n+1}))
\end{split}
\end{equation}
for $f\in Alt_A^n(\mathcal{L},M),~\mbox{and}~ x_i\in L, ~\mbox{for}~ 1\leq i\leq n+1$.

\begin{Prop}
If $f\in Alt_A^n(\mathcal{L},M)$, then $\delta f \in Alt_A^{n+1}(\mathcal{L},M)$ and $\delta^2=0$. 
\end{Prop}

\begin{proof}
Here, we need to check that for any $f\in Hom_A(\wedge_A^n L,M)$, 
$\delta f\in Hom_A(\wedge_A^{n+1} L,M) $ and $\delta^2=0$. This follows using the fact that $(\rho,\phi)$ is a representation of the hom-Lie algebra $(L,[-,-],\alpha)$ on $A$ and $(M,\beta)$ is a left $(\mathcal{L},\alpha)$-module.
\end{proof}
In view of this proposition we obtain that $(Alt_A(\mathcal{L},M),\delta)$ is a cochain complex. We denote the cohomology space of the cochain complex $(Alt_A(\mathcal{L},M),\delta)$ by $H^*_{hR}(\mathcal{L},M)$.

In particular, if we consider a regular hom-Lie algebra simply as a hom-Lie-Rinehart algebra, then the above cochain complex with coefficients in a hom-Lie algebra module is same as the cochain complex mentioned in Section 3 of \cite{hom-big}. 

Moreover, a hom-Lie algebroid is a particular case of hom-Lie-Rinehart algebras. Thus by following the above discussion, we define representations of a hom-Lie algebroid, and subsequently define a cohomology of a hom-Lie algebroid with coefficients in a representation in the next section. 
\begin{Exm} 
If $\alpha=Id_L$, then the hom-Lie- Rinehart algebra $(\mathcal{L},\alpha)$ is simply a Lie-Rinehart algebra $L$ over $A$ and the algebra $A$ is a Lie-Rinehart algebra module over $L$. Also, the cohomology $H^*_{hR}(\mathcal{L}, M)$ is same as the Lie-Rinehart algebra cohomology with coefficients in the module $A$.
\end{Exm}

\begin{Exm} Let $M$ be a smooth manifold, $A=C^{\infty}(M)$ and $L=\chi(M)$- the space of smooth vector fields on $M$. Let $\alpha=Id_L$, then the cohomology $H^*_{hR}(L,A)$ is the de-Rham cohomology of $M$.
\end{Exm}

Let $(\mathcal{L},\alpha)$ be a regular hom-Lie-Rinehart algebra over $(A,\phi)$. Then the pair $(A,\phi)$ is a left $(\mathcal{L},\alpha)$-module. Consider the cochain complex $(Alt_A(\mathcal{L},A),\delta)$ of the hom-Lie-Rinehart algebra $(\mathcal{L},\alpha)$ with coefficients in $(A,\phi)$. Now, define a multiplication $\wedge:Alt_A(\mathcal{L},A)\otimes_R Alt_A(\mathcal{L},A)\rightarrow Alt_A(\mathcal{L},A)$
$$(\xi\wedge\eta)(x_1,\cdots,x_{p+q})=\sum_{\sigma\in S(p,q)}sgn(\sigma)(\xi(x_{\sigma(1)},\cdots,x_{\sigma(p)}).\eta(x_{\sigma(p+1)},\cdots,x_{\sigma(p+q)})).$$
for $x_1,\cdots,x_{p+q}\in L$, $\xi\in Alt_A^p(\mathcal{L},A)$, and $\eta\in Alt_A^q(\mathcal{L},A)$. Here, $S(p,q)$ denotes the set of all $(p,q)$-shuffles in the symmetric group $S_{p+q}$. Also, define $\Phi:Alt_A(\mathcal{L},A)\rightarrow Alt_A(\mathcal{L},A)$ as follows
$$\Phi(\xi)(x_1,\cdots,x_p)=\phi(\xi(\alpha^{-1}(x_1),\cdots,\alpha^{-1}(x_p)))$$
for $x_1,\cdots,x_{p}\in L$, and $\xi\in Alt_A^p(\mathcal{L},A)$. The multiplication $\wedge$ makes the chain complex $Alt_A(\mathcal{L},A)$, a graded commutative algebra since $A$ is a commutative algebra. Next, we have the following result:
 
\begin{Lem}\label{lemhlr}
Let $(\mathcal{L},A)$ be a regular hom-Lie-Rinehart algebra, where the underlying $A$-module $L$ is projective of rank $n$. Then the pair $(\wedge^n_A L^*,\Phi)$ is a right $(\mathcal{L},\alpha)$-module.
\end{Lem}
\begin{proof}
For any $x\in L$, define the contraction $i_x:Alt_A(\mathcal{L},A)\rightarrow Alt_A(\mathcal{L},A)$ by
$$i_x(\eta)(x_1,\cdots,x_{m-1})=\Phi(\eta)(\alpha(x),x_1,\cdots,x_{m-1})$$
for $x_1,\cdots,x_{m-1}\in L$ and $\eta\in Alt_A^m(\mathcal{L},A)$. By a straightforward calculation, for any $\eta_1\in Alt_A^m(\mathcal{L},A)$, and $\eta_2\in Alt_A^k(\mathcal{L},A)$, we get the following equation:
\begin{equation}\label{intprop1}
i_x(\eta_1\wedge\eta_2)=i_x(\eta_1)\wedge \Phi(\eta_2)+(-1)^m \Phi(\eta_1)\wedge i_x(\eta_2)
\end{equation}

Define the right action $\Theta:\wedge^n_A L^*\times L\rightarrow \wedge^n_A L^*$ as follows:
\begin{equation}\label{rightaction}
\Theta(\xi,x)=-d(i_{\alpha^{-1}(x)}(\Phi^{-1}(\xi))).
\end{equation}
for all $X\in\wedge^n_A L$ and $\xi\in \wedge^n_A L^*$. Then it follows that the pair $(\wedge^n_A L^*,\Phi)$ is a right $(\mathcal{L},\alpha)$-module.
\end{proof}

\begin{Thm}\label{Corres2}
Let $(\mathcal{L},\alpha)$ be a regular hom-Lie-Rinehart algebra over $(A,\phi)$. If $L$ is a projective $A$-module of rank $n$, then there is a bijective correspondence between right $(\mathcal{L},\alpha)$-module structures on $(A,\phi)$ and left $(\mathcal{L},\alpha)$-module structures on $(\wedge^n_A L,\alpha)$.

\begin{proof}
First assume that $(A,\phi)$ is a right $(\mathcal{L},\alpha)$-module and the right action is given by $(a,x)\mapsto a.x$. Here, $\wedge^n_A L$ is projective $A$-module of rank $1$, so we have an isomorphism of $A$-modules $\Psi:\wedge^n_A L\rightarrow Hom_A(\wedge^n_A L^*,A)$ given by 
$$\Psi(X)(\xi)=\xi(X)$$
for $X\in \wedge^n_A L$ and $\xi\in  \wedge^n_A L^*$. For each $X\in \wedge^n_A L$, denote $\Psi(X)$ by $\Psi_X$ and define an invertible map $\gamma: Hom_A(\wedge^n_A L^*,A)\rightarrow Hom_A(\wedge^n_A L^*,A)$ as follows
$$\gamma(a.\Psi_X)=\phi(a).\Psi_{\alpha(X)}$$ 
for each $X\in \wedge^n_A L$.

Now, for any $x\in L$, $X\in \wedge_A^n L$, and $\xi\in \wedge_A^n L^*$, denote $\Phi^{-1}(\xi)$ by $\bar{\xi}$ and define a left action $\nabla:L\times Hom_A(\wedge^n_A L^*,A)\rightarrow Hom_A(\wedge^n_A L^*,A)$ as follows
$$\nabla(x,\Psi_X)(\xi)=\gamma(\Psi_X)(\Theta(\bar{\xi},x))-(\Psi_X(\bar\xi)).x$$
Here the map $\Theta:\wedge^n_A L^*\times L\rightarrow \wedge^n_A L^*$ is the right action of the $A$-module $L$ on $\wedge^n_A L^*$, as defined in Lemma \ref{lemhlr}. By direct calculation, the left action $\nabla$ makes the pair $(Hom_A(\wedge^n_A L^*,A),\gamma)$ a left $(\mathcal{L},\alpha)$-module. Subsequently, it gives a left $(\mathcal{L},\alpha)$-module structure on $(\wedge^n_A L,\alpha)$.

Conversely, let us consider $(\wedge^n_A L,\alpha)$ is a left $(\mathcal{L},\alpha)$-module, where the left action is given by $(x,X)\mapsto \nabla_x(X)$. Since $\wedge^n_A L$ is a projective $A$-module of rank $1$, we have an isomorphism $\theta:A\rightarrow Hom_A(\wedge^n_A L,\wedge^n_A L)$ which maps $a\mapsto \theta_a$, where $\theta_a(X)=\theta(a)(X)=a.X$ for $a\in A,~X\in\wedge^n_A L$. The isomorphism $\phi:A\rightarrow A$ induces an invertible map $\bar{\phi}:Hom_A(\wedge^n_A L,\wedge^n_A L)\rightarrow Hom_A(\wedge^n_A L,\wedge^n_A L)$, defined as: $\bar{\phi}(\theta_a)=\theta_{\phi(a)}.$

Define a right action $\mu:Hom_A(\wedge^n_A L,\wedge^n_A L)\times L\rightarrow Hom_A(\wedge^n_A L,\wedge^n_A L)$ by:
$$\mu(\theta_a,x)(X)=\bar{\phi}(\theta_a)([x,\alpha^{-1}(X)]_{\mathfrak{G}})-\nabla_x(\theta_a(\alpha^{-1}(X)))$$
here, $X\in \wedge^n_A L,~ x\in L,~a\in A$ and $[-,-]_{\mathfrak{G}}$ is the hom- Gerstenhaber bracket obtained from the hom-Lie bracket on $L$. A straightforward but long calculation shows that the pair $( Hom_A(\wedge^n_A L,\wedge^n_A L),\bar{\phi})$ is a right $(\mathcal{L},\alpha)$-module. Hence, there is a right  $(\mathcal{L},\alpha)$-module structure on $(A,\phi)$.
\end{proof}
  
\end{Thm}  

Now, by Theorem \ref{Corres1} and Theorem \ref{Corres2}, we immediately get the following result:

\begin{Cor}\label{Res1}
Let $(\mathcal{L},\alpha)$ be a regular hom-Lie-Rinehart algebra over $(A,\phi)$.  If $L$ is a projective $A$-module of rank $n$, then there is a bijective correspondence between exact generators of the hom-Gerstenhaber algebra bracket on $\wedge^*_A L$ and left $(\mathcal{L},\alpha)$-module structures on $(\wedge^n_A L,\alpha)$.
\end{Cor}

\section{Applications to hom-Lie algebroids}

 Here we consider hom-Lie algebroids as defined  by Laurent-Gengoux and Teles in \cite{hom-Lie}. There is also a modified version of hom-Lie algebroids presented by L. Cai, et al.  in \cite{hom-Lie1} and an equivalence is shown for the case of regular (or invertible) hom-Lie algebroids. First, we define Representations of a hom-Lie algebroid.
\subsection{\textbf{Representations of hom-Lie algebroids}} Let $\mathcal{A}:=(A,\phi, [-,-],\rho,\alpha)$ be a hom-Lie algebroid and $(E,\phi,\beta)$ be a hom-bundle over a smooth manifold $M$. A bilinear map $\nabla: \Gamma A\otimes \Gamma E\rightarrow \Gamma E$, denoted by $\nabla(x,s):=\nabla_x(s)$, is a representation of $\mathcal{A}$ on a hom-bundle $(E,\phi,\beta)$ if it satisfies the following properties:  
\begin{enumerate}
\item $\nabla_{f.x}(s)=\phi^*(f).\nabla_x(s)$;
\item $\nabla_x(f.s)=\phi^*(f).\nabla_x(s)+ \rho(x)[f].\beta(s)$;
\item $(\nabla,\beta)$ is a representation of the underlying hom-Lie algebra $(\Gamma A,[-,-],\alpha)$ on $\Gamma E$;
\end{enumerate}
for all $x\in \Gamma A,~s\in \Gamma E$ and $f\in C^{\infty}(M)$.

\begin{Exm}\label{Trivial Representation}
Let $\mathcal{A} =(A,\phi, [-,-],\rho,\alpha)$ be a hom-Lie algebroid over $M$. Suppose $\phi^*:C^{\infty}(M)\rightarrow C^{\infty}(M)$ denotes the algebra homomorphism induced by the smooth map $\phi:M\rightarrow M$. Define a map $$\nabla^{\phi^*}: \Gamma A\otimes C^{\infty}(M)\rightarrow C^{\infty}(M)$$
given by $\nabla^{\phi^*}(x,f)=\rho(x)[f]$ for $x\in \Gamma A$ and $f\in C^{\infty}(M)$. Then $\nabla^{\phi^*}$ is a canonical representation of $\mathcal{A}$ on the hom-bundle $(M\times \mathbb{R}, \phi,\phi^*)$.
\end{Exm}

\begin{Exm}\label{ExistenceR}
Let $\mathcal{A} =(A,\phi, [-,-],\rho,\alpha)$ be a hom-Lie algebroid over $M$ and $(E,\phi,\beta)$ be a hom-bundle over $M$, where $E$ is a trivial line bundle over $M$. Assume $s\in \Gamma E$ is a nowhere vanishing section of the trivial line bundle $E$ over $M$. Define a map $\nabla: \Gamma A\otimes \Gamma E \rightarrow \Gamma E$ by 
$$\nabla(x,f.s)=\rho(x)[f].\beta(s)$$
for all $x\in \Gamma A$ and $f\in C^{\infty}(M)$. Then the map $\nabla$ is a representation of $\mathcal{A}$ on $(E,\phi,\beta)$. 
\end{Exm}

Let $\mathcal{A}=(A,\phi, [-,-],\rho,\alpha)$ be a hom-Lie algebroid, where $A$ is a vector bundle of rank $n$ over $M$, then $\wedge^n A$ is a line bundle over $M$. Extend the map $\alpha: \Gamma A\rightarrow\Gamma A$ to a map $\tilde{\alpha}:\Gamma(\wedge^n A)\rightarrow (\Gamma\wedge^n A)$ defined by 
$$\tilde{\alpha}(x_1\wedge\cdots \wedge x_n)=\alpha(x_1)\wedge\cdots\wedge\alpha(x_n)$$
for any $x_1,x_2,\cdots,x_n\in \Gamma A$. 

\begin{Prop}\label{Corres3}
Let $\mathcal{A}=(A,\phi, [-,-],\rho,\alpha)$ be a regular hom-Lie algebroid. Then there is a one-one correspondence between representations of $\mathcal{A}$ on the hom-bundle $(\wedge^n A,\phi,\tilde{\alpha})$ and exact generators of the associated hom-Gerstenhaber algebra $\mathfrak{A}:=(\oplus_{k\geq 0}\Gamma\wedge^kA^*,\wedge,[-,-]_{\mathcal{A}},\tilde{\alpha})$ (here, $\tilde{\alpha}$ is  extension of the map $\alpha$ to higher degree elements).
\end{Prop}

\begin{proof}
This result follows from the Corollary \ref{Res1}. More precisely, given an exact generator $D$ of the associated hom-Gerstenhaber algebra $\mathfrak{A}$, define a map $\nabla:\Gamma A\otimes \Gamma(\wedge^n A)\rightarrow \Gamma(\wedge^n A)$ by
$$\nabla(a,X)=[a,X]_{\mathcal{A}}-(Da)\tilde{\alpha}(X),$$
where $a\in A$, and $X\in \Gamma(\wedge^n A)$. It is immediate to check that $\nabla$ is a representation of $\mathcal{A}$ on the hom-bundle $(\wedge^n A,\phi,\tilde{\alpha})$.

Conversely, let $\nabla$ is a representation of $\mathcal{A}$ on the hom-bundle $(\wedge^n A,\phi,\tilde{\alpha})$, then there exists a unique generator $D$ of the hom-Gerstenhaber bracket such that for any $X\in \Gamma(\wedge^n A)$ the following condition is satisfied:
\begin{equation}\label{eqcond}
  D(a)\tilde{\alpha}(X)=[a,X]_{\mathcal{A}}-\nabla(a,X).
\end{equation}
Define $D$ on higher degree elements by the following relation:
$$D(a\wedge b)=-[a,b]_{\mathcal{A}}+D(a)\tilde{\alpha}(b)-\tilde{\alpha}(a)\wedge D(b),$$
for $a\in A$, and $b\in \Gamma(\wedge^k A)$. By using this relation and the fact that $a\wedge X=0$ for any $X\in \Gamma(\wedge^n A)$, the following condition is equivalent to equation \eqref{eqcond}:
$$\nabla(a,X)=-\tilde{\alpha}(a)\wedge D(X).$$
 
\end{proof}

Let $\mathcal{A}=(A,\phi,[-,-],\rho,\alpha)$ be a hom-Lie algebroid and $(E,\phi,\beta)$ be a hom-bundle over a smooth manifold $M$, where $E$ is a line bundle. Then the following proposition extracts a representation of $\mathcal{A}$ on the hom-bundle $(E,\phi,\beta)$ from a given representation of $\mathcal{A}$ on the square hom-bundle given by the triplet $(E^2:= E \otimes E,\phi,\bar{\beta}:=\beta\otimes \beta)$ over $M$. 

\begin{Prop}\label{sqrt-ModSt}
Let $\mathcal{A}:=(A,  \phi,[-,-],\rho,\alpha)$ be a hom-Lie algebroid, and the triplet $(E,\phi,\beta)$ be a hom-bundle over a smooth manifold $M$, where $E$ is a line bundle. If the map $ \bar{\nabla}$ is a representation of $\mathcal{A}$ on the hom-bundle $(E^2,\phi,\bar{\beta})$. Then the map $\nabla:\Gamma A\otimes \Gamma E\rightarrow \Gamma E$ is a representation of $\mathcal{A}$ on the hom-bundle $(E,\phi,\beta)$, which is defined as follows: 

Let $s$ be a section of  $E$ and $U$ be an open subset of $M$, then $s=f.t$ for some $f\in C^{\infty}(U)$ and some section $t\in\Gamma M$, which vanishes nowhere over $U$. Then define 
$$\nabla(x,s)\big|_{\phi^{-1}(U)}=(\rho(x)(f).\beta(t))\big|_{\phi^{-1}(U)}+\frac{1}{2}(\bar{\nabla}(x,t^2)/\bar{\beta}(t^2)).\beta(s)\big|_{\phi^{-1}(U)}.$$

Furthermore, the map $\nabla:\Gamma A\otimes \Gamma E\rightarrow \Gamma E$ gives back the initial map $\bar{\nabla}:\Gamma A\otimes \Gamma (E\otimes E)\rightarrow \Gamma (E\otimes E)$ by the following equation:
$$\bar{\nabla}(x,s_1\otimes s_2)=\nabla(x,s_1)\otimes \beta(s_2)+\beta(s_1)\otimes \nabla(x,s_2)$$
for $x\in \Gamma A$, $s_1,s_2\in \Gamma E$.
\end{Prop}

\begin{proof}
First we prove that $(\nabla,\beta)$ is a representation of $(\Gamma A,[-,-],\alpha)$ on $\Gamma E$. We need to show the following:
\begin{itemize}
\item $\nabla([x,y],\beta(s))=\nabla(\alpha(x),\nabla(y,s))-\nabla(\alpha(y),\nabla(x,s)),$ and 
\item $\nabla(\alpha(x),\beta(s))=\beta(\nabla(x,s)).$
\end{itemize}

Let $s$ be a section of $E$ and $U\subset M$ be an open subset of $M$, then $s=f.t$ for some $f\in C^{\infty}(U)$ and some section $t\in\Gamma M$, which vanishes nowhere over $U$. Then
\begin{equation}\label{part3}
\nabla([x,y],\beta(s))\Big|_{\phi^{-2}(U)}=(\rho([x,y])(\phi(f)).\beta^2(t))\Big|_{\phi^{-2}(U)}+\frac{1}{2}(\bar{\nabla}([x,y],\bar{\beta}(t^2))/\bar{\beta}^2(t^2)).\beta^2(s)\Big|_{\phi^{-2}(U)}
\end{equation}
Here,
$$(\rho([x,y])(\phi(f)).\beta^2(t))\Big|_{\phi^{-2}(U)}=\Big(\rho(\alpha(x))\rho(y)(f)-\rho(\alpha(y))\rho(x)(f)\Big).\beta^2(t)\Big|_{\phi^{-2}(U)}.$$
Since $\beta$ is a bijective map, the section  $\bar{\beta}(t^2)$ vanishes nowhere over the open subset $\phi^{-1}(U)$ of $M$, so we can write
\begin{equation}\label{Relation}
{\nabla}(x,t^2)|_{\phi^{-1}(U)}=f_{x}.\bar{\beta}(t^2)|_{\phi^{-1}(U)}
\end{equation}
for some $f_{x}\in C^{\infty}(M)$, and $x\in\Gamma A$. The map $ \bar{\nabla}$ is a representation of $\mathcal{A}$ on the hom-bundle $(E^2,\phi,\bar{\beta})$, i.e.
$$\bar{\nabla}([x,y],\bar{\beta}(t^2))=\bar{\nabla}(\alpha(x),\bar{\nabla}(y,t^2))-\bar{\nabla}(\alpha(y),\bar{\nabla}(x,t^2))~~
\mbox{and}$$ 
$$\bar{\nabla}(\alpha(x),\bar{\beta}(t^2))=\bar{\beta}(\bar{\nabla}(x,t^2)).$$
Thus,
\begin{equation}\label{Relation1}
\phi^*(f_{\alpha^{-1}([x,y])})=\rho(\alpha(x))(f_y)-\rho(\alpha(y))(f_x).
\end{equation}
Also, by the definition of $\nabla$, we have the following:
\begin{align}\label{pt1}
\nonumber
&\nabla\big(\alpha(x),\nabla(y,s)\big)\big|_{\phi^{-2}(U)}\\\nonumber
=&~\nabla\Big(\alpha(x),\big(\rho(y)(f).\beta(t)+\frac{1}{2}(\bar{\nabla}(y,t^2)/\bar{\beta}(t^2)).\beta(s)\big)\Big)\Big|_{\phi^{-2}(U)}\\\nonumber
=&~\Big(\rho(\alpha(x))\rho(y)(f).\beta^2(t)+\frac{1}{2}\frac{\bar{\nabla}(\alpha(x),\bar{\beta}(t^2))}{\bar{\beta}^2(t^2)}.\beta\big(\rho(y)(f).\beta(t)\big)\\\nonumber
&+\frac{1}{2}\rho(\alpha(x))\Big(\frac{\bar{\nabla}(y,t^2)}{\bar{\beta}(t^2)}.\phi^*(f)\Big).\beta^2(t)+\frac{1}{4}\frac{\bar{\nabla}(\alpha(x),\bar{\beta}(t^2))}{\bar{\beta}^2(t^2)}.\beta\Big(\frac{\bar{\nabla}(y,t^2)}{\bar{\beta}(t^2)}.\beta(s)\Big)\Big)\Big|_{\phi^{-2}(U)}\\
=&~\Big(\rho(\alpha(x))\rho(y)(f).\beta^2(t)+\frac{1}{2}\frac{\bar{\nabla}(\alpha(x),\bar{\beta}(t^2))}{\bar{\beta}^2(t^2)}\rho(\alpha(y))(\phi^*(f)).\beta^2(t)\\\nonumber
&+\frac{1}{2}\rho(\alpha(x))\Big(\frac{\bar{\nabla}(y,t^2)}{\bar{\beta}(t^2)}.\phi^*(f)\Big).\beta^2(t)+\frac{1}{4}\frac{\bar{\nabla}(\alpha(x),\bar{\beta}(t^2))}{\bar{\beta}^2(t^2)}\phi^*\Big(\frac{\bar{\nabla}(y,t^2)}{\bar{\beta}(t^2)}\Big).\beta^2(s)\Big)\Big|_{\phi^{-2}(U)}.\\\nonumber
\end{align}
By using equation \eqref{Relation}, the equation \eqref{pt1} can also be expressed as:
\begin{align}\label{part1}
\nabla\big(\alpha(x),\nabla(y,s)\big)\Big|_{\phi^{-2}(U)}=&~\Big(\rho(\alpha(x))\rho(y)(f).\beta^2(t)+\frac{1}{2}\phi^*(f_{x}).\rho(\alpha(y))(\phi^*(f)).\beta^2(t)\\\nonumber
&+\frac{1}{2}\rho(\alpha(x))\big(f_{y}.\phi^*(f)\big).\beta^2(t)+\frac{1}{4}\phi^*(f_{x}.f_y).\beta^2(s)\Big)\Big|_{\phi^{-2}(U)}.\\\nonumber
\end{align}
Similarly,
\begin{align}\label{part2}
\nabla(\alpha(y),\nabla(x,s)\big)\Big|_{\phi^{-2}(U)}=&~\Big(\rho(\alpha(y))\rho(x)(f).\beta^2(t)+\frac{1}{2}\phi^*(f_{y}).\rho(\alpha(x))(\phi^*(f)).\beta^2(t)\\\nonumber
&+\frac{1}{2}\rho(\alpha(y))\big(f_{x}.\phi^*(f)\big).\beta^2(t)+\frac{1}{4}\phi^*(f_{x}.f_y).\beta^2(s)\Big)\Big|_{\phi^{-2}(U)}.\\\nonumber
\end{align}
By equations \eqref{part3}, \eqref{Relation1}, \eqref{part1}, and \eqref{part2}, we get the following equation:
\begin{equation}\label{item1}
\nabla\big(\alpha(x),\nabla(y,s)\big)-\nabla(\alpha(y),\nabla(x,s)\big)
=\nabla([x,y],\beta(s)).
\end{equation}
Moreover, 
\begin{align*}
&\nabla(\alpha(x),\beta(s))\big|_{\phi^{-2}(U)}\\
=&~\Big(\rho(\alpha(x))(\phi^*(f)).\beta^2(t) +\frac{1}{2} \frac{\bar{\nabla}(\alpha(x),\bar{\beta}(t^2))}{\bar{\beta}^2(t^2)}.\beta^2(s)\Big)\Big|_{\phi^{-2}(U)}\\
=&~\Big(\beta\big(\rho(x)(f).\beta(t)\big)+\frac{1}{2} \phi^*(f_{x}).\beta^2(s)\Big)\Big|_{\phi^{-2}(U)}\\
=&~\Big(\beta\big(\rho(x)(f).\beta(t)+\frac{1}{2} f_{x}.\beta(s)\big)\Big)\Big|_{\phi^{-2}(U)}\\
=&~\beta(\nabla(x,s))\big|_{\phi^{-2}(U)},
\end{align*}
i.e.,
\begin{equation}\label{item2}
\nabla(\alpha(x),\beta(s))=\beta(\nabla(x,s)).
\end{equation}
Thus by equations \eqref{item1}, and \eqref{item2}, the pair $(\nabla,\beta)$ is a representation of $(\Gamma A,[-,-],\alpha)$ on $\Gamma E$. It is immediate to observe that:
\begin{itemize}
\item $\nabla(f.x,s)=\phi^*(f).\nabla(x,s),$
\item $\nabla(x,f.s)=\phi^*(f).\nabla(x,s)+\rho(x)(f).\beta(s),$
\end{itemize}
for all $x\in \Gamma A$, $f\in C^{\infty}(M)$, and $s\in \Gamma E$. Therefore, $\nabla$ is a representation of $\mathcal{A}$ on the hom-bundle $(E,\phi,\beta)$. Let $U\subset M$ be an open subset of $M$, then 
$$\bar{\nabla}(x,s_1\otimes s_2)|_{\phi^{-1}(U)}=\nabla(x,s_1)\otimes \beta(s_2)|_{\phi^{-1}(U)}+\beta(s_1)\otimes \nabla(x,s_2)|_{\phi^{-1}(U)}$$
for $x\in \Gamma A$, $s_1,s_2\in \Gamma E$. Hence, squaring the map $\nabla:\Gamma A\otimes \Gamma E\rightarrow \Gamma E$ gives back the original map $\bar{\nabla}:\Gamma A\otimes \Gamma (E^2)\rightarrow \Gamma (E^2)$.
\end{proof}

The Proposition \ref{sqrt-ModSt} generalises Proposition $4.2$ of \cite{Trmsr}, which states: If a representation of a Lie algebroid on the square of a line bundle is given then there exists a representation of the Lie algebroid on the line bundle.

\begin{Rem}\label{Existence of Mod}

 Let $A$ be a real vector bundle of rank $n$ over a manifold $M$. If $\mathcal{A}:=(A,\phi, [-,-], \rho,\alpha)$ is a hom-Lie algebroid over $M$ then $\wedge^n A$ is a real line bundle over $M$. Note that $\wedge^n A\otimes \wedge^n A$ is a trivial line bundle over $M$. Now, define a map $\bar{\alpha}:\wedge^n A\otimes \wedge^nA\rightarrow \wedge^n A\otimes \wedge^n A$ as follows:
$$\tilde{\alpha}(X\otimes Y)=\tilde{\alpha}(X)\otimes \tilde{\alpha}(Y)$$
for $X,Y\in \wedge^n A$. By Example \ref{ExistenceR}, there exists a representation of $\mathcal{A}$ on the hom-bundle $(\wedge^n A\otimes \wedge^n A, \phi,\tilde{\alpha})$. Consequently, by the Proposition \ref{sqrt-ModSt} we get a representation of $\mathcal{A}$ on the hom-bundle $(\wedge^n A,\phi,\tilde{\alpha})$.
\end{Rem}
\subsection{\textbf{Cohomology of regular hom-Lie algebroids}} 

Let $\mathcal{A}:=(A,\phi, [-,-], \rho,\alpha)$ be a regular hom-Lie algebroid over $M$ and the map $\nabla$ be a representation of $\mathcal{A}$ on the hom-bundle $(E,\phi,\beta)$. We define a cochain complex $(C^*(\mathcal{A};E),d_{A,E})$
for $\mathcal{A}$ with coefficients in this representation as follows:

$C^*(\mathcal{A};E):=\oplus_{n\geq 0} \Gamma(Hom(\wedge^n A,E))$, and the map $d_{A,E}:\Gamma(Hom(\wedge^n A,E))\rightarrow \Gamma(Hom(\wedge^{n+1} A,E))$ is defined as
\begin{equation}\label{differential}
\begin{split}
(d_{A,E} \Xi)(x_1,\cdots,&x_{n+1})= \sum_{i=1}^{n+1}(-1)^{i+1}\nabla_{(\alpha^{-1}(x_i))}(\Xi(\alpha^{-1}(x_1),\cdots,\hat{x_i},\cdots,\alpha^{-1}(x_{n+1}))\\&+\sum_{1\leq i<j\leq n+1}(-1)^{i+j}\beta(\Xi(\alpha^{-2}([x_i,x_j]),\alpha^{-1}(x_1),\cdots,\hat{x_i},\cdots,\hat{x_j},\cdots,\alpha^{-1}(x_{n+1}))
\end{split}
\end{equation}
where $\Xi\in \Gamma(Hom(\wedge^n A,E)), x_i\in \Gamma A ~\mbox{and}~~ 1\leq i\leq n+1$. By using the definition of a representation, it follows that the map $d_{A,E}$ is a well-defined square zero operator. We denote the cohomology of the resulting cochain complex $(C^*(\mathcal{A};E),d_{A,E})$ by $H^*(\mathcal{A},E)$. In particular, if $\alpha=Id_A$ and $\phi=Id_M$, then $\mathcal{A}$ is a Lie algebroid and $H^*(\mathcal{A},E)$ is the usual de-Rham cohomology of the Lie algebroid $\mathcal{A}$ with coefficients in the representation on the vector bundle $E$.

We now consider the trivial representation of the regular hom-Lie algebroid $\mathcal{A}$ on the trivial hom-bundle $(M\times \mathbb{R},\phi,\phi^*)$ given in Example \ref{Trivial Representation}. Then the cochain complex $(C^*(\mathcal{A};M\times \mathbb{R}),d_{A,M\times \mathbb{R}})$ gives the map $d_{A,M\times \mathbb{R}}:\Gamma(\wedge^nA^*)\rightarrow \Gamma(\wedge^{n+1}A^*)$ defined by
\begin{equation}\label{NCBD}
\begin{split}
(d_{A,M\times \mathbb{R}} \xi)(x_1,\cdots,&x_{n+1})= \sum_{i=1}^{n+1}(-1)^{i+1}\rho{(\alpha^{-1}(x_i))}[\xi(\alpha^{-1}(x_1),\cdots,\hat{x_i},\cdots,\alpha^{-1}(x_{n+1})]\\&+\sum_{1\leq i<j\leq n+1}(-1)^{i+j}\phi^*(\xi(\alpha^{-2}([x_i,x_j]),\alpha^{-1}(x_1),\cdots,\hat{x_i},\cdots,\hat{x_j},\cdots,\alpha^{-1}(x_{n+1}))
\end{split}
\end{equation}
for $\xi\in \Gamma(\wedge^nA^*),~\mbox{and}~ x_i\in \Gamma A, ~\mbox{for}~ 1\leq i\leq n+1$. Let us denote the coboundary map $d_{A,M\times \mathbb{R}}$ simply by $d_A$. Define a map $\hat{\alpha}:\Gamma(\wedge^n A^*)\rightarrow \Gamma(\wedge^{n}A^*)$ by
$$\hat{\alpha}(\xi)(x_1,\cdots,x_n)=\phi^*(\xi(\alpha^{-1}(x_1),\cdots,\alpha^{-1}(x_n)))$$
for $\xi\in \Gamma(\wedge^nA^*),~\mbox{and}~ x_i\in \Gamma A, ~\mbox{for}~ 1\leq i\leq n$.
Then we have the following result:

\begin{Thm}\label{dgca1}
Let $(A,\phi,\alpha)$ be a regular hom-bundle over $M$, i.e. the map $\phi:M\rightarrow M$ is a diffeomorphism and $\alpha:\Gamma A\rightarrow \Gamma A$ is an invertible map. Then a hom-Lie algebroid structure $\mathcal{A}:=(A,\phi, [-,-], \rho,\alpha)$ on the hom-bundle $(A,\phi,\alpha)$ is equivalent to a $(\hat{\alpha},\hat{\alpha})$-differential graded commutative algebra
 on $\oplus_{n\geq 0}\Gamma(\wedge^nA^*)$.
 \end{Thm}
\begin{proof}
Let $\mathcal{A}:=(A,\phi, [-,-], \rho,\alpha)$ be a hom-Lie algebroid over $M$. Consider the coboundary operator $d_A$ defined by equation (\ref{NCBD}). Then 
we need to prove that 
$$d_A(\zeta\wedge \eta)=d_A(\zeta)\wedge \Psi(\eta)+(-1)^{|\zeta|}\Psi(\zeta)\wedge d_A(\eta)$$
for all $\zeta\in \Gamma(\wedge^p A^*)$, $\eta\in \Gamma(\wedge^q A^*)$, and $p,q\geq 0$. It simply follows by induction on the degree $p$. Then the tuple $(\oplus_{n\geq 0}\Gamma\wedge^nA^*,\hat{\alpha},d_A)$ is a $(\hat{\alpha},\hat{\alpha})$-differential graded commutative algebra.

Conversely, let $(\oplus_{n\geq 0}\Gamma\wedge^nA^*,\hat{\alpha},d)$ be a $(\hat{\alpha},\hat{\alpha})$-differential graded commutative algebra. We define the anchor map $\rho$ given by 
$\rho(x)[f]=<d_Af,\alpha(x)>,$
and the hom-Lie bracket $[-,-]$ is given by
\begin{equation*}
\begin{split}
<[x,y],\xi>=\rho(\alpha^2(x))<\tilde{\alpha}^{-1}(\xi),y>-\rho(\alpha^2(y))<\tilde{\alpha}^{-1}(\xi),x> -(d(\tilde{\alpha}^{-1}(\xi)))(\alpha(x),\alpha(y))
\end{split}
\end{equation*}
for all $x,y\in\Gamma A$, $f\in C^{\infty}(M)$ and $\xi\in\Gamma A^*$. Then it follows that $(A,\phi, [-,-], \rho,\alpha)$ is a hom-Lie algebroid.
\end{proof}
A version of the above theorem is proved in \cite{hom-Lie1}, by considering a modified definition of hom-Lie algebroid and the associated cochain complex. Let us recall the following definitions of interior multiplication and Lie derivative from \cite{hom-Lie1}:
\begin{itemize}
\item For any $X\in \Gamma(\wedge^k A)$, define the interior multiplication $i_X:\Gamma(\wedge^n A^*)\rightarrow \Gamma(\wedge^{n-k} A^*)$ by
$$(i_X\Xi)(x_1,x_2,\cdots,x_{n-k})=(\hat{\alpha}(\Xi))(\tilde{\alpha}(X),x_1,\cdots,x_{n-k})$$
for any $x_1,x_2,\cdots,x_{n-k}\in \Gamma A$.
\item Let $X\in \Gamma(\wedge^k A)$, then define the Lie derivative $L_X:\Gamma(\wedge^n A^*)\rightarrow \Gamma(\wedge^{n-k+1} A^*)$ by
$$L_X\circ\hat{\alpha}=i_X\circ d_A-(-1)^k d_A \circ i_{(\tilde{\alpha})^{-1}(X)}.$$

\end{itemize}
Now for all $X\in\Gamma(\wedge^k A)$, $Y\in\Gamma(\wedge^l A)$, $f\in C^{\infty}(M)$, and $\Xi\in \Gamma(\wedge^n A^*)$ the interior multiplication satisfies the following properties:
\begin{enumerate}
\item $i_{fX}\Xi=i_{X}(f.\Xi)=\phi^*(f).i_X,$
\item $\hat{\alpha}(i_X(\Xi))=i_{\tilde{\alpha}(X)}(\hat{\alpha}(\Xi)),$
\item $i_{\tilde{\alpha}(X\wedge Y)}\circ \hat{\alpha}=i_{\tilde{\alpha}(Y)}\circ i_X=(-1)^{kl}i_{\tilde{\alpha}(X)}\circ i_Y,$
\end{enumerate}
The Lie derivative $L_X:\Gamma(\wedge^n A^*)\rightarrow \Gamma(\wedge^{n-k+1} A^*)$, for any $X\in \Gamma(\wedge^k A)$ satisfies the following equation:
$$L_{f.X}\Xi=\phi^*(f).L_X\Xi-(-1)^k d_A f\wedge i_X(\Xi)$$ for all $f\in C^{\infty}(M)$, and $\Xi\in \Gamma(\wedge^n A^*)$.
If $k=1$, i.e. $X\in \Gamma A$, then for all $f\in C^{\infty}(M)$, and $\Xi\in \Gamma(\wedge^n A^*)$, we have
$$L_{X}(f.\Xi)=\phi^*(f).L_X\Xi+\rho(x)(f)\hat{\alpha}(\Xi)$$
Moreover, for $x\in \Gamma A$, the interior multiplication $i_x$ and Lie derivative $L_x$ satisfy the following identities:
 $$i_x(\Xi_1\wedge \Xi_2)=i_x\Xi_1\wedge \hat{\alpha}(\Xi_2)+(-1)^m\hat{\alpha}(\Xi_1)\wedge i_x\Xi_2,$$
 $$L_x(\Xi_1\wedge \Xi_2)=L_x\Xi_1\wedge \hat{\alpha}(\Xi_2)+\hat{\alpha}(\Xi_1)\wedge L_x\Xi_2,$$
for any $\Xi_1\in \Gamma(\wedge^m A^*)$, $\Xi_2\in \Gamma(\wedge^n A^*)$. The above properties and identities follows from \cite{hom-Lie1}, by replacing the differential $d:\Gamma(\wedge^k A^*)\rightarrow \Gamma(\wedge^{k+1} A^*)$, defined in Section $3$ of \cite{hom-Lie1}, by the differential $d_A$ given by equation \eqref{NCBD}.
 
\begin{Rem}\label{hom-Poisson}
For a manifold $M$ with a diffeomorphism $\psi:M\rightarrow M$, the tangent hom-Lie algebroid $\mathcal{T}$ is given by the tuple $(\psi^!TM,\psi,[-,-]_{\psi^*}, Ad_{\psi^*},Ad_{\psi^*})$, where the bracket $[-,-]_{\psi^*}$ and anchor map $Ad_{\psi^*}$ are defined as follows:
\begin{itemize}
\item $[X,Y]_{\psi^*}=\psi^*\circ X\circ (\psi^*)^{-1}\circ Y\circ (\psi^*)^{-1}-\psi^*\circ Y\circ (\psi^*)^{-1}\circ X\circ (\psi^*)^{-1}$ for $X,Y\in\psi^!TM $;
\item $Ad_{\psi^*}(X)=\psi^*\circ X\circ (\psi^*)^{-1}$ for $X\in\psi^!TM $.
\end{itemize}

If $(M,\psi,\pi)$ is a hom-Poisson manifold, where $\psi:M\rightarrow M$ is a diffeomorphism and $\pi$ is a hom-Poisson bivector, then $\mathcal{T^*}:=(\psi^{!}T^{*} M,\psi,[-,-]_{\pi^{\#}},{Ad_{\psi^*}}^{\dagger},\pi^{\#}\circ Ad_{\psi^*}^{\dagger})$ is the cotangent hom-Lie algebroid as defined in \cite{hom-Lie1}, where 

\begin{itemize}
\item $[\xi,\eta]_{\pi^{\#}}=L_{\pi^{\#}(\xi)}(\eta)-L_{\pi^{\#}(\eta)}(\xi)-d\pi(\xi,\eta)$ for $\xi,\eta\in \psi^{!}T^{*} M$,
\item ${Ad_{\psi^*}}^{\dagger}(\xi)(X)=\psi^*(\xi (Ad_{\psi^*}^{-1}(X)))$ for $\xi\in \psi^{!}T^{*} M,~X\in\psi^!TM $, and
\item the anchor map is $\pi^{\#}\circ Ad_{\psi^*}^{\dagger}$  instead of $\pi^{\#}$ (Note that $\pi^{\#}$ is the anchor map for the cotangent hom-Lie algebroid in \cite{hom-Lie1}), since we are using the Definition \ref{hom-Lie} of hom-Lie algebroid.
\end{itemize}

Let us consider $(\Gamma(\wedge^{top} \psi^{!}T^{*} M))^2:=\Gamma(\wedge^{top} \psi^{!}T^{*} M)\otimes \Gamma(\wedge^{top} \psi^{!}T^{*} M)$, where $\wedge^{top}$ denotes the highest exterior power. Define a map $D: \Gamma(\psi^{!}T^{*} M)\otimes (\Gamma(\wedge^{top} \psi^{!}T^{*} M))^2\rightarrow (\Gamma(\wedge^{top} \psi^{!}T^{*} M))^2$ as follows: 
$$D_{\xi}(\eta_1\otimes \eta_2):= D(\xi,(\eta_1\otimes \eta_2))=[[\xi,\eta_1]]_{\pi^{\#}} \otimes {Ad_{\psi^*}}^{\dagger}(\eta_2)+{Ad_{\psi^*}}^{\dagger}(\eta_1)\otimes L_{\pi^{\#}(\xi)}(\eta_2)  $$
for $\xi\in \Gamma(\psi^{!}T^{*} M)$ and $\eta_1,\eta_2\in \Gamma(\wedge^{top}\psi^{!}T^{*} M)$. Then it follows by a long but straightforward calculation that the map $D$ is a representation of $\mathcal{T^*}$ on the hom-bundle $((\Gamma(\wedge^{top} \psi^{!}T^{*} M))^2,\psi,\bar{Ad_{\psi^*}^{\dagger}})$, where the map $\bar{Ad_{\psi^*}^{\dagger}}$ is the extension of the map $Ad_{\psi^*}^{\dagger}$ to $(\Gamma(\wedge^{top} \psi^{!}T^{*} M))^2$. 
\end{Rem}

\subsection{\textbf{Homology of regular hom-Lie algebroids}}
Let $\mathcal{A}:=(A,\phi, [-,-], \rho,\alpha)$ be a regular hom-Lie algebroid over a manifold $M$, where $A$ is a real vector bundle of rank $n$ over $M$, then by Remark \ref{Existence of Mod}, we get a representation of $\mathcal{A}$ on the hom-bundle $(\wedge^n A,\phi,\tilde{\alpha})$, given by the map  $\nabla:\Gamma A\otimes \Gamma \wedge^n A\rightarrow \Gamma \wedge^n A$. Next, by Proposition \ref{Corres3}, we get a chain complex $(\oplus_{k\geq 0}\Gamma\wedge^kA,D_{\nabla})$. Thus we get a homology of regular hom-Lie algebroid $\mathcal{A}$ given by the homology of the chain complex $(\oplus_{k\geq 0}\Gamma\wedge^kA,D_{\nabla})$. Let us denote this homology by $H_{*}^{\nabla}(\mathcal{A})$.

If $\nabla^1$ and $\nabla^2$ are two representations of $\mathcal{A}$ on the hom-bundle $(\wedge^n A,\phi,\tilde{\alpha})$, then it is natural to ask about the relation between the induced homologies.

Firstly, let $D_{\nabla^1}$ and $D_{\nabla^2}$ be exact generators of the associated hom-Gerstenhaber algebra $\mathfrak{A}$ to the hom-Lie algebroid $\mathcal{A}$, obtained by Proposition \ref{Corres3}, respectively. Then first observe that 
$$D_{\nabla^1}-D_{\nabla^2}(f.x)=\phi(f).(D_{\nabla^1}-D_{\nabla^2})(x)$$
for $f\in C^{\infty}(M)$ and $x\in \Gamma A$. Therefore, there exists $\xi\in \Gamma A^*$ such that 
\begin{equation}\label{gen-diff1}
D_{\nabla^1}-D_{\nabla^2}(x)=\phi(\xi(x))=i_{\xi}(x),
\end{equation}
for $x\in \Gamma A$. Since $D_{\nabla^1}$ and $D_{\nabla^2}$ commute with the map $\alpha$, we have $\phi^*(i_{\xi}(x))=i_{\xi}(\alpha(x))$. By equation \eqref{generator}, $D_{\nabla^1}-D_{\nabla^2}$ satisfies the following condition:
$$(D_{\nabla^1}-D_{\nabla^2})(X\wedge Y)=(D_{\nabla^1}-D_{\nabla^2})(X)\wedge\tilde{\alpha}(Y)+(-1)^{|X|}\tilde{\alpha}(Y)\wedge (D_{\nabla^1}-D_{\nabla^2})(Y)$$
for $X,Y\in \mathfrak{A}$, $|X|$ denotes degree of $X$. Then for any $X\in\mathfrak{A}$, it follows that
$D_{\nabla^1}-D_{\nabla^2}(X)=i_{\xi}(X)$
 and $\tilde{\alpha}\circ i_{\xi}=i_{\xi}\circ \tilde{\alpha}$. The equation \eqref{gen-diff1} is equivalent to the following equation: 
\begin{equation}
(\nabla^1_x-\nabla^2_x)(X)=i_{\xi}(x).\tilde{\alpha}(X)
\end{equation}
for $x\in \Gamma A$ and $X\in \Gamma (\wedge^n A)$. Now, let us make the following observations:
\begin{itemize}
\item $(D_{\nabla^2}\circ i_{\xi}+i_{\xi}\circ D_{\nabla^2})(x,y)=-i_{d\xi}(\alpha(x),\alpha(y))$ for any $x,y\in \Gamma A$, and
\item by using the fact that $D_{\nabla^1}\circ D_{\nabla^1}=0$, and $D_{\nabla^2}\circ D_{\nabla^2}=0$, we get $d\xi=0$, i.e. $\xi$ is a $1$-cocycle.
\end{itemize}

We say that the maps $\nabla^1$ and $\nabla^2$ are homotopic if $\xi$ is a $1$-coboundary, i.e. $\xi=df$ for some $f\in C^{\infty}(M)$. Similarly, in this case the corresponding generators $D_{\nabla^1}$ and $D_{\nabla^2}$ are said to be homotopic. 

\begin{Thm}
If representations $\nabla^1$ and $\nabla^2$ of $\mathcal{A}$ on the hom-bundle $(\wedge^n A,\phi,\tilde{\alpha})$ are homotopic, then $H_{*}^{\nabla^1}(\mathcal{A})\cong H_{*}^{\nabla^2}(\mathcal{A})$.  
\end{Thm}

\begin{proof}
It is given that maps $\nabla^1$ and $\nabla^2$ are homotopic, i.e. for $X\in\Gamma(\wedge^{top}A)$ and $x\in \Gamma A$ we have the following equation  
$$(\nabla^1_x-\nabla^2_x)(X)=i_{df}(x).\tilde{\alpha}(X),$$
or equivalently 
$$D_{\nabla^1}-D_{\nabla^2}=i_{df}$$
for some $f\in C^{\infty}(M)$ such that $\tilde{\alpha}\circ i_{df}=i_{df}\circ\tilde{\alpha}$. Now, Let us define a map $F:\oplus_{k\geq 0}\Gamma\wedge^kA\rightarrow\oplus_{k\geq 0}\Gamma\wedge^kA$ by
$$F(\lambda)=e^f\lambda$$
for $\lambda\in \Gamma\wedge^kA,~k\geq 0$. By using the fact that $[e^f,\lambda]_{\mathcal{A}}=-\rho(\lambda)(e^f)=e^{\phi^*(f)}[f,\lambda]_{\mathcal{A}}$, and equation \eqref{generator} for generators $D_{\nabla_1}$ and $D_{\nabla_2}$, we get the following relation:
\begin{equation}\label{rel-gen}
D_{\nabla_1}(e^f.\lambda)=e^{\phi^*(f)}.D_{\nabla_2}(\lambda)
\end{equation}
for any $\lambda\in\Gamma\wedge^kA,~k\geq 0 $. Now, it is clear that $F(Ker(D_{\nabla^2}))\subset Ker(D_{\nabla^1})$ and thus $F$ induces a map $\tilde{F}:Ker(D_{\nabla^2})\rightarrow Ker(D_{\nabla^1})$. 
Let $\lambda=D_{\nabla^2}(R)$ for $\lambda\in\Gamma(\wedge^kA),~k>0$ and for some $R\in \Gamma(\wedge^{k-1}A),~k>0$. Then 
$$F(\lambda)=e^f.\lambda=e^f.D_{\nabla^2}(R)=D_{\nabla^1}(e^{{\phi^*}^{-1}(f)}.R),$$
Hence $\tilde{F}$ induces an isomorphism $\mathcal{F}:H_{*}^{\nabla^2}(\mathcal{A})\rightarrow H_{*}^{\nabla^1}(\mathcal{A})$. 
\end{proof}

Let $(M,\psi,\pi)$ be a hom-Poisson manifold and recall the associated cotangent hom-Lie algebroid from remark \ref{hom-Poisson}, which is given by $\mathcal{T^*}:=(\psi^{!}T^{*} M,\psi,[-,-]_{\pi^{\#}},{Ad_{\psi^*}}^{\dagger},\pi^{\#}\circ Ad_{\psi^*}^{\dagger})$. Also recall that the map $D$ is a representation of $\mathcal{T^*}$ on the hom-bundle $((\wedge^{top} \psi^{!}T^{*} M)^2, \psi, ({Ad_{\psi^*}}^{\dagger})^2)$. Then by Proposition \ref{sqrt-ModSt}, we have a representation of $\mathcal{T^*}$ on the hom-bundle $(\wedge^{top} \psi^{!}T^{*} M,\psi,{Ad_{\psi^*}}^{\dagger})$. In particular, we have the following result:

\begin{Prop}\label{hom-Poisson R1}
Let $(M,\psi,\pi)$ be a hom-Poisson manifold. Define a map $$\bar{D}:\Gamma(\psi^{!}T^{*} M)\otimes \Gamma(\wedge^{top} \psi^{!}T^{*} M)\rightarrow \Gamma(\wedge^{top} \psi^{!}T^{*} M)$$ as follows:
\begin{equation}\label{repontop}
\bar{D}(\xi,\mu)=[\xi,\mu]_{\pi^{\#}}-\pi(d\xi).{Ad_{\psi^*}}^{\dagger}(\mu)
\end{equation}
for any $\xi\in \Gamma(\psi^{!}T^{*} M) $, $\mu\in \Gamma(\wedge^{top} \psi^{!}T^{*} M)$. Then the map $\bar{D}$ is a representation of $\mathcal{T^*}$ on the hom-bundle $(\wedge^{top} \psi^{!}T^{*} M,\psi,{Ad_{\psi^*}}^{\dagger})$.
\end{Prop}
\begin{proof}
Let us denote $\tilde{\gamma}:={Ad_{\psi^*}}^{\dagger}(\gamma)$, and $\bar{\gamma}:=({Ad_{\psi^*}}^{\dagger})^{-1}(\gamma)$ for $\gamma\in \Gamma(\wedge^k \psi^{!}T^{*} M)$. Then 
$$\bar{D}([\xi,\eta]_{\pi^{\#}},\tilde{\mu})=[[\xi,\eta]_{\pi^{\#}},\tilde{\mu}]_{\pi^{\#}}-\pi(d[\xi,\eta]_{\pi^{\#}})\tilde{\mu}$$

Also, by using equation \eqref{repontop}, we have the following:
\begin{align}\label{Rep1}
\nonumber\bar{D}(\tilde{\xi},\bar{D}(\eta,\mu))=&~\bar{D}(\tilde{\xi},[[\eta,\mu]]_{\pi^{\#}}-\pi(d\eta)\tilde{\mu})\\\nonumber
=&~[[\tilde{\xi},[[\eta,\mu]]_{\pi^{\#}}]]_{\pi^{\#}}-\pi(d\tilde{\xi}).[[\tilde{\eta},\tilde{\mu}]]_{\pi^{\#}}-[[\tilde{\xi},\pi(d\eta)\tilde{\mu}]]_{\pi^{\#}}\\\nonumber
&+\pi(d\tilde{\xi}).\psi^*(\pi(d\eta))\tilde{\tilde{\mu}}\\
=&~[[\tilde{\xi},[[\eta,\mu]]_{\pi^{\#}}]]_{\pi^{\#}}-\pi(d\tilde{\xi}).[[\tilde{\eta},\tilde{\mu}]]_{\pi^{\#}}-\psi^*(\pi(d\eta))[[\tilde{\xi}, \tilde{\mu}]]_{\pi^{\#}}\\\nonumber
&-\pi^{\#}(\tilde{\tilde{\xi}})(\pi(d\eta)).\tilde{\tilde{\mu}}+\pi(d\tilde{\xi}).\psi^*(\pi(d\eta))\tilde{\tilde{\mu}}
\end{align}

Similarly,
\begin{align}\label{Rep2}
 \bar{D}(\tilde{\eta},\bar{D}(\xi,\mu))
=&~[[\tilde{\eta},[[\xi,\mu]]_{\pi^{\#}}]]_{\pi^{\#}}-\pi(d\tilde{\eta}).[[\tilde{\xi},\tilde{\mu}]]_{\pi^{\#}}-\psi^*(\pi(d\xi))[[\tilde{\eta}, \tilde{\mu}]]_{\pi^{\#}}\\\nonumber
&-\pi^{\#}(\tilde{\tilde{\eta}})(\pi(d\xi)).\tilde{\tilde{\mu}}+\pi(d\tilde{\eta}).\psi^*(\pi(d\xi))\tilde{\tilde{\mu}}
\end{align}
Now let us observe the following: 

\begin{itemize}
\item $L_X(\gamma(\alpha))=L_X(\gamma)(Ad_{\psi^*}\alpha)+\tilde{\gamma}
(L_X(\alpha))$ for all $X\in \Gamma\psi^{!}{TM}$, $\alpha\in \Gamma(\wedge^k\psi^{!}{TM})$, and $\gamma\in \Gamma(\wedge^k\psi^{!}{T^*M})$. 
\item $L_{Ad_{\psi^*}(X)}(d\xi)=dL_{X}\xi$ for all $X\in \Gamma\psi^{!}{TM}$, and $\xi\in \Gamma(\wedge\psi^{!}{T^*M})$.
\end{itemize}
Using above observations and the fact that $Ad_{\psi^*}\circ \pi^{\#}=\pi^{\#}\circ Ad_{\psi^*}^{\dagger}$, we have the following :
\begin{align}\label{Rep3}
\nonumber
&-\pi^{\#}(\tilde{\tilde{\xi}})(\pi(d\eta))+\pi^{\#}(\tilde{\tilde{\eta}})(\pi(d\xi))\\\nonumber
&=-[[\pi^{\#}(\tilde{{\xi}}),\pi]]_{\psi^*}(d\tilde{\eta})-Ad_{\psi^*}\circ\pi(L_{\pi^{\#}(\tilde{{\xi}})}d\eta)+[[\pi^{\#}(\tilde{{\eta}}),\pi]]_{\psi^*}(d\tilde{\xi})+Ad_{\psi^*}\circ\pi(L_{\pi^{\#}(\tilde{{\eta}})}d\xi)\\\nonumber
&=[[\pi^{\#}(\tilde{{\eta}}),\pi]]_{\psi^*}(d\tilde{\xi})-[[\pi^{\#}(\tilde{{\xi}}),\pi]]_{\psi^*}(d\tilde{\eta})+\pi(L_{\pi^{\#}({\tilde{\eta}})}d\xi-L_{\pi^{\#}(\tilde{{\xi}})}d\eta)\\\nonumber
&=[[\pi^{\#}(\tilde{{\eta}}),\pi]]_{\psi^*}(d\tilde{\xi})-[[\pi^{\#}(\tilde{{\xi}}),\pi]]_{\psi^*}(d\tilde{\eta})+\pi( d(L_{\pi^{\#}({\eta})}d\xi-L_{\pi^{\#}({\xi})}\eta+d\pi(\xi,\eta)))\\
&=-\pi(d[[\xi,\eta]]_{\pi^{\#}}).
\end{align}
Thus, by using equations \eqref{repontop}, \eqref{Rep1}, \eqref{Rep2}, and \eqref{Rep3}, we immediately get the following:
\begin{itemize}
\item $\bar{D}([\xi,\eta]_{\pi^{\#}},\tilde{\mu})=\bar{D}(\tilde{\xi},\bar{D}(\eta,\mu))-\bar{D}(\tilde{\eta},\bar{D}(\xi,\mu))$;
\item $\bar{D}(\tilde{\xi},\tilde{\mu})=Ad_{\psi^{*}}^{\dagger}\bar{D}(\xi,\mu);$
\end{itemize}
i.e. the pair $(\bar{D},{Ad_{\psi^*}}^{\dagger})$ is a representation of hom-Lie algebra $(\Gamma(\psi^{!}T^{*} M),[-,-]_{\pi^{\#}},{Ad_{\psi^*}}^{\dagger})$ on the line bundle $\Gamma(\wedge^{top} \psi^{!}T^{*} M)$ with respect to the map ${Ad_{\psi^*}}^{\dagger}$. Furthermore, $\bar{D}(f.\xi,\mu)=\psi^*(f).\bar{D}(\xi,\mu)$ and $\bar{D}(\xi,f.\mu)=\psi^*(f).\bar{D}(\xi,\mu)+\rho^{\#}(\tilde{\xi})(f).\tilde{\mu}$ for any $f\in C^{\infty}(M)$, $\xi\in \Gamma(\wedge\psi^{!}{T^*M})$, and $\mu\in \Gamma(\wedge^{top}\psi^{!}{T^*M})$. Hence, the map $\bar{D}$ is a representation of $\mathcal{T^*}$ on the hom-bundle $(\wedge^{top} \psi^{!}T^{*} M,\psi,{Ad_{\psi^*}}^{\dagger})$.
\end{proof}
\begin{Rem}
Note that Equation \eqref{repontop} can be rewritten as 
\begin{align*}
\bar{D}(\xi,\mu)
&=L_{\pi^{\#}(\xi)}(\mu)+\pi(d\xi).\tilde{\mu}\\
&=\tilde{\xi}\wedge d i_{\pi}(\bar{\mu})
\end{align*}
\end{Rem}
By Proposition \ref{Corres3}, the representation $\bar{D}$ of $\mathcal{T}^*$ on hom-bundle $(\wedge^{top} \psi^{!}T^{*} M,\psi,{Ad_{\psi^*}}^{\dagger})$ corresponds to an exact generator of the hom-Gerstenhaber algebra associated to the cotangent hom-Lie algebroid $\mathcal{T}^*$. In particular, we get the operator 
$[i_{\pi},d]_{Ad_{\psi^*}^{\dagger}}:\Gamma(\wedge^k \psi^{!}T^{*} M)\rightarrow \Gamma(\wedge^{k-1} \psi^{!}T^{*} M)$, defined as
\begin{align*}
[i_{\pi},d]_{Ad_{\psi^*}^{\dagger}}&=Ad_{\psi^*}^{\dagger}\circ i_{\pi}\circ (Ad_{\psi^*}^{\dagger})^{-1}\circ d\circ (Ad_{\psi^*}^{\dagger})^{-1}-Ad_{\psi^*}^{\dagger}\circ d\circ (Ad_{\psi^*}^{\dagger})^{-1}\circ i_{\pi}\circ (Ad_{\psi^*}^{\dagger})^{-1}
\end{align*}

We call the homology of the chain complex $(\oplus_{k\geq }\Gamma(\wedge^k \psi^{!}T^{*} M),[i_{\pi},d]_{Ad_{\psi^*}^{\dagger}})$, the \textbf{hom-Poisson homology} associated to hom-Poisson manifold $(M,\psi,\pi)$. Note that in the case $\psi=id_M$, it gives the Poisson homology.

\section{Strong differential hom-Gerstenhaber algebras}

Let $(A,\psi,\alpha)$ be invertible hom-bundle over $M$, $(A,\psi,[-,-]_A,\rho_A,\alpha)$ and $(A^*,\psi, [-,-]_{A^*},\rho_{A^*},{\alpha^{\dagger}})$ be two Hom-Lie algebroids in duality (here, $\alpha^{\dagger}(\xi)(X)=\psi^*(\xi(\alpha^{-1}(X)))$ for $\xi\in \Gamma A^*$ and $X\in \Gamma A$). Then recall from \cite{hom-Lie1}  that the pair $(A,A^*)$ is called a hom-Lie bialgebroid if
\begin{equation}\label{dercond}
d_*[x,y]_A = [d_*x,\alpha(y)] + [\alpha(x),d_*y],~~~ \mbox{for all}~x,y\in \Gamma A.
\end{equation} 
 
 Here, the map $d_*$ is the coboundary map given by \eqref{differential} for the hom-Lie algebroid $(A^*,\psi,[-,-]_{A^*},\\\rho_{A^*},{\alpha^{\dagger}})$ and the bracket on the right hand side is the hom-Gerstenhaber bracket induced on exterior bundle over $A$ by the hom-Lie algebroid structure $(A,\psi,[-,-]_A,\rho_A,\alpha)$.
\begin{Def}
A differential hom-Gerstenhaber algebra is a hom-Gerstenhaber algebra $\mathfrak{A}:=(\oplus_{i\in\mathbb{Z}_+}\mathcal{A}_i,\wedge,[-,-],\alpha)$ equipped with a degree $1$ map $d:\mathfrak{A}\rightarrow \mathfrak{A}$ such that 
\begin{itemize}
\item the map $d$ is a $(\alpha,\alpha)$-derivation of degree $1$ with respect to the graded commutative and associative product $\wedge$, i.e.
$d(X\wedge Y)=d(X)\wedge\alpha(Y)+(-1)^{|X|}\alpha(X)\wedge d(Y)$ for $X,Y\in \mathfrak{A}$.
\item $d^2=0$ and the map $d$ commutes with $\alpha$, i.e. $d\circ \alpha=\alpha\circ d$.
\end{itemize}

The hom-Gerstenhaber algebra $\mathfrak{A}$ is said to be a \textbf{strong differential hom-Gerstenhaber algebra} if $d$ also satisfies the equation:
$d[X,Y]=[dX,\alpha(Y)]+[\alpha(X),dY]$
for $X,Y\in \mathfrak{A}$. Let us denote this strong differential hom-Gerstenhaber algebra by the tuple $(\oplus_{i\in\mathbb{Z}_+}\mathcal{A}_i,\wedge,[-,-],\alpha,d)$. If the map $\alpha:\mathfrak{A}\rightarrow \mathfrak{A} $ is an invertible map, then we say $\mathfrak{A}$ is a strong differential regular hom-Gerstenhaber algebra.
\end{Def}

\begin{Exm}\label{sdga1}
Let $(\mathfrak{g},[-,-]_{\mathfrak{g}},\alpha)$ and $(\mathfrak{g}^*,[-,-]_{\mathfrak{g}^*},\alpha^{\dagger})$ be two regular hom-Lie algebras (where $\alpha^{\dagger}(\xi)(X)=\xi(\alpha^{-1}(X))$ for $\xi\in L^*$ and $X\in L$). Recall from \cite{hom-Liebi}, $(L,L^*)$ is a purely hom-Lie bialgebra if the following compatibility condition holds:
\begin{equation}\label{bialcon}
\Delta([x,y]_L)=[\alpha^{-1}(x),\Delta(y)]_{\mathfrak{G}}+[\Delta(x),\alpha^{-1}(y)]_{\mathfrak{G}}.
\end{equation}

Here, $x,y\in L$, the bracket $[-,-]_{\mathfrak{G}}$ is the hom-Gerstenhaber bracket obtained by extending $[-,-]_{\mathfrak{g}}$, and $\Delta:L\rightarrow \wedge^2 L$ is the dual map of the hom-Lie algebra bracket $[-,-]_{\mathfrak{g}^*}:\wedge^2\mathfrak{g}^*\rightarrow \mathfrak{g}^*$. 

Let us first recall from Section $3$ of \cite{hom-Liebi} that the adjoint representation of $(\mathfrak{g},[-,-]_{\mathfrak{g}},\alpha)$ on $\wedge^k \mathfrak{g}$ is given by $ad_x(Y)=[x,Y]_\mathfrak{G}$
for $x\in \mathfrak{g}$, $Y\in\wedge^k \mathfrak{g}$; and $[-,-]_{\mathfrak{G}}$ is the associated hom-Gerstenhaber bracket on the exterior algebra $\wedge^*\mathfrak{g}$. Then the above equation \eqref{bialcon} can also be expressed as
\begin{equation}\label{bialcon2}
\Delta([x,y]_L)=ad_{\alpha^{-1}(x)}\Delta(y)-ad_{\alpha^{-1}(y)}\Delta(x).
\end{equation}
Also, recall that the co-adjoint representation of $(\mathfrak{g},[-,-]_{\mathfrak{g}},\alpha)$ on $\mathfrak{g}^*$ with respect to $\alpha^{\dagger}=(\alpha^{-1})^*$ is given by the map $ad^*:\mathfrak{g}\rightarrow \mathfrak{g}l(\mathfrak{g}^*)$ such that
$<ad_x^*(\xi),y>=-<\xi,ad_x(y)>, $
and $ad_x^*(\xi)=ad_{\alpha(x)}^*((\alpha^{-2})^{*}\xi)$
for $x,y\in \mathfrak{g}$ and $\xi\in \mathfrak{g}^*$. The map $ad_x^*:\mathfrak{g}^*\rightarrow \mathfrak{g}^*$ for any $x\in\mathfrak{g}$ is extended to higher degree elements by the following equation: 
\begin{equation}\label{co-ad rep}
ad_x^{*}(\xi_1\wedge\cdots\wedge \xi_n)=\sum_{1\leq i\leq n}\alpha^*(\xi_1)\wedge ad_x^*(\xi_i)\wedge\cdots\wedge \alpha^*(\xi_n)
\end{equation}
where $\xi_i\in\mathfrak{g}^*$ for all $1\leq i\leq n$. Let $d_*$ be the coboundary operator of the hom-Lie algebra $(\mathfrak{g}^*,[-,-]_{\mathfrak{g}^*},\alpha^{\dagger})$ given by equation (\ref{NCBD}) (Consider the hom-Lie algebra $(\mathfrak{g}^*,[-,-]_{\mathfrak{g}^*},\alpha^{\dagger})$ as a hom-Lie algebroid over point manifold $M$). Note that for any $x,y\in \mathfrak{g}$, and $\eta,\xi\in \mathfrak{g}^*$ we have
$$d_{*}[x,y](\xi\wedge \eta)=-[x,y]((\alpha^{-2})^*([\xi,\eta]))=-<\Delta(\alpha^2[x,y]),\xi\wedge \eta>.$$ 
Then by using equations \eqref{bialcon2}, and \eqref{co-ad rep}, we get the following derivation condition:
$$d_*[X,Y]_{\mathfrak{G}}=[d_*X,\alpha_{\mathfrak{G}}(Y)]_{\mathfrak{G}}+[\alpha_{\mathfrak{G}}(X),d_*Y]_{\mathfrak{G}}.$$
for any $X,Y\in\wedge^*\mathfrak{g}$. Thus the hom-Gerstenhaber algebra $(\mathfrak{G}=\wedge^*\mathfrak{g},\wedge,[-,-]_{\mathfrak{G}},\alpha_{\mathfrak{G}})$ associated to the hom-Lie algebra $(\mathfrak{g},[-,-]_{\mathfrak{g}},\alpha)$ gives a strong differential hom-Gerstenhaber algebra with differential $d_*$. 

Moreover, considering the coboundary operator $d$ of the hom-Lie algebra $(\mathfrak{g},[-,-]_{\mathfrak{g}},\alpha)$ given by \eqref{NCBD}, the hom-Gerstenhaber algebra $(\mathfrak{G}^*=\wedge^*\mathfrak{g}^*,\wedge,[-,-]_{\mathfrak{G}}^*,\alpha_{\mathfrak{G}^*})$ associated to the hom-Lie algebra $(\mathfrak{g}^*,[-,-]_{\mathfrak{g}^*},\alpha^{\dagger})$ gives a strong differential hom-Gerstenhaber algebra with the strong differential $d:\mathfrak{G}^*\rightarrow \mathfrak{G}^*$.
\end{Exm}

\begin{Exm} Suppose $(M,\psi,\pi)$ is a hom-Poisson manifold and $\psi:M\rightarrow M$ is a diffeomorphism. Then we have  the tangent hom-Lie algebroid $(\psi^!TM,\psi,[-,-]_{\psi^*}, Ad_{\psi^*},Ad_{\psi^*})$ and the cotangent hom-Lie algebroid $(\psi^{!}T^{*} M,\psi,[-,-]_{\pi^{\#}},{Ad_{\psi^*}}^{\dagger},\pi^{\#}\circ Ad_{\psi^*}^{\dagger})$ respectively. Let us denote the corresponding differentials of tangent and cotangent hom-Lie algebroids with trivial representation by $d$ and $d_*,$ respectively. Then we have the following observations:
 
\begin{itemize}
\item For any $X\in \Gamma\wedge^k(\psi^!T^*M)$,  $d_*X=[[\pi,X]]_{\psi^*}$, where $[[-,-]]_{\psi^*}$ is the hom-Gerstenhaber bracket obtained by extending the bracket $[-,-]_{\psi^*}$. By using the graded hom-Jacobi identity for $[[-,-]]_{\psi^*}$, we obtain the following equation:
$$d_*[x,y]_{\psi^*}=[[d_*x,Ad_{\psi^*}(y)]]_{\psi^*}+[[Ad_{\psi^*}(x),d_*y]]_{\psi^*}$$
for any $x,y\in \Gamma\psi^!TM$. Finally, by using the hom-Leibniz rule for the bracket $[[-,-]]_{\psi^*}$ with respect to the graded commutative product, one obtains the following derivation condition:
$$d_*[[X,Y]]_{\psi^*}=[[d_*X,\bar{Ad_{\psi^*}}(Y)]]_{\psi^*}+[[\bar{Ad_{\psi^*}}(X),d_*Y]]_{\psi^*}$$
for any $X,Y\in \oplus_{i\in\mathbb{Z}_+}\Gamma \wedge^i(\psi^!TM)$. Thus the hom-Gerstenhaber algebra associated to tangent hom-Lie algebroid, is a strong differential hom-Gerstenhaber algebra with the differential $d$.

\item Hom-Gerstenhaber algebra $\mathfrak{A}^*=(\oplus_{i\in\mathbb{Z}_+}\Gamma \wedge^i(\psi^!T^*M),\wedge,[[-,-]]_{\pi^{\#}},\bar{Ad_{\psi^*}^{\dagger}})$ (the map $\bar{Ad_{\psi^*}^{\dagger}}$ is the extension of the map $Ad_{\psi^*}^{\dagger}$ to the exterior bundle) associated to the cotangent hom-Lie algebroid, is a strong differential hom-Gerstenhaber algebra with the differential: $d_*.$
\end{itemize}  
\end{Exm}

\begin{Exm}\label{SDHGCOMP}
Given a Gerstenhaber algebra $(\mathcal{A},[-,-],\wedge)$ with a strong differential $d$ and an endomorphism $\alpha:(\mathcal{A},[-,-],\wedge)\rightarrow (\mathcal{A},[-,-],\wedge)$ satisfying $d\circ \alpha=\alpha\circ d$, then first note that the quadruple $(\mathcal{A},\wedge,[-,-]_{\alpha}=\alpha\circ [-,-],\alpha)$ is a hom-Gerstenhaber algebra. Let us define $d_{\alpha}:=\alpha\circ d$, then $d_{\alpha}:\mathcal{A}\rightarrow\mathcal{A}$ is a map of degree $1$. Since the differential $d$ is a derivation of degree $1$ with respect to graded commutative product $\wedge$, the map $d_{\alpha}$ satisfies the following derivation condition with respect to $\wedge$: 
$$d_{\alpha}(X\wedge Y)=d_{\alpha}(X)\wedge \alpha(Y)+(-1)^{|X|}\alpha(X)\wedge d_{\alpha}(Y)$$
for all $X,Y\in \mathcal{A}$. Furthermore, the map $d_\alpha$ satisfies square zero condition and for all $X,Y\in \mathcal{A}$ it satisfies the following derivation condition with respect to $[-,-]_{\alpha}$:
$$d_{\alpha}[X,Y]_{\alpha}=[d_{\alpha}(X),\alpha(Y)]_{\alpha}+[\alpha(X),d_{\alpha}(Y)]_{\alpha}.$$
Thus the above hom-Gerstenhaber algebra $(\mathcal{A},\wedge,[-,-]_{\alpha}=\alpha\circ [-,-],\alpha)$ is a strong differential hom-Gerstenhaber algebra with strong differential $d_{\alpha}$.
\end{Exm}

Let $A$ be a vector bundle over $M$, $\psi:M\rightarrow M$ be a smooth map, $\beta:\Gamma A\rightarrow \Gamma A$ be a linear map then the space $\oplus_{i\in\mathbb{Z}_+}\Gamma(\wedge^i A)$ has a hom-Gerstenhaber algebra structure if and only if we have a hom-Lie algebroid structure on hom-bundle $(A, \psi, \beta)$, see Theorem 4.4, \cite{hom-Lie}. 

 Now, let $(A,\psi,\beta)$ be a hom-bundle, where $\psi: M\rightarrow M$ is a diffeomorphism, $\beta:\Gamma A\rightarrow \Gamma A$ is a bijective linear map. Also let $\mathfrak{A}=(\oplus_{i\in\mathbb{Z}_+}\Gamma(\wedge^iA),\wedge,[-,-],\alpha,d)$ be a strong differential regular hom-Gerstenhaber algebra structure on the space of multisections $\oplus_{i\in\mathbb{Z}_+}\Gamma(\wedge^iA)$, where $\alpha:\mathfrak{A}\rightarrow \mathfrak{A}$ is a degree $0$ map such that $\alpha_0=\psi^*:C^{\infty}(M)\rightarrow C^{\infty}(M)$ is an algebra automorphism induced by $\psi$, $\alpha_1=\beta$ and the map $\alpha$ is an extension of $\alpha_0$ and $\alpha_1$ to the higher sections. 

\begin{Thm}
The tuple $(\oplus_{i\in\mathbb{Z}_+}\Gamma(\wedge^iA),\wedge,[-,-],\alpha,d)$ is a strong differential regular hom-Gersten-haber algebra if and only if $(A,A^*)$ is a hom-Lie bialgebroid.

\begin{proof}
Let us assume that $\mathfrak{A}=(\oplus_{i\in\mathbb{Z}_+}\Gamma(\wedge^iA),\wedge,[-,-],\alpha,d)$ is a strong differential regular hom-Gerstenhaber algebra. It is clear that $(A,\psi,[-,-]_A,\rho_A,\alpha_1)$ is a hom-Lie algebroid, where $[-,-]_A$ and $\rho_A$ are obtained by restricting $[-,-]$ on $\Gamma A\times \Gamma A$ and $\Gamma A\times C^{\infty}(M)$, respectively (Theorem 4.4, \cite{hom-Lie}). Since $(\oplus_{i\in\mathbb{Z}_+}\Gamma(\wedge^iA),\alpha,d)$ is an $(\alpha, \alpha)$-differential graded algebra, by using Theorem \ref{dgca1}, we obtain the bracket $[-,-]_{A^*}$ on $ A^*$ and the anchor map $\rho_{A^*}$ such that $(A^*,\phi, [-,-]_{A^*},\rho_{A^*},\beta^{\dagger})$ is a hom-Lie algebroid (here, $\beta^{\dagger}(\xi)(X)=\psi(\xi(\beta^{-1}(X)))$ for $\xi\in\Gamma A^*$ and $X\in \Gamma A$). Note that the differential corresponding to this hom-Lie algebroid structure, given by equation \eqref{NCBD}, coincides with the differential $d$. Moreover, the derivation property of the differential $d$ with respect to the Lie bracket $[-,-]_A$ on $A$, i.e. the equation: 
$$d[X,Y]_A=[dX,\alpha_1(Y)]+[\alpha_1(X),dY].$$
for $X,Y\in \Gamma A$ implies that $(A,A^*)$ is a hom-Lie bialgebroid. 
 
Conversely, for a given hom-Lie bialgebroid $(A,A^*)$ we obtain a hom-Gerstenhaber algebra structure given by $\mathfrak{A}=(\oplus_{i\in\mathbb{Z}_+}\Gamma(\wedge^iA),\wedge,[-,-],\alpha)$ since there is a hom-Lie algebroid structure on the hom-bundle $(A,\psi,\beta) $ over $M$. Let $d$ be the differential given by equation \eqref{NCBD}, for the hom-Lie algebroid structure on $(A^*,\psi,\beta^{\dagger})$. Then $\mathfrak{A}$ is an $(\alpha,\alpha)$-differential graded commutative algebra with the differential $d$. Now, by equation \eqref{dercond}, it is clear that $\mathfrak{A}=(\oplus_{i\in\mathbb{Z}_+}\Gamma(\wedge^iA),\wedge,[-,-],~\alpha,d)$ is a strong differential regular hom- Gerstenhaber algebra.
\end{proof}
\end{Thm}

\begin{Exm}
Let $(M,\psi,\pi)$ be a given hom-Poisson manifold where the smooth map $\psi:M\rightarrow M$ is a diffeomorphism, then there are hom-Lie algebroid structures on both the pull back bundles $\psi^{!}TM$ and $\psi^{!}T^*M$. Let $\mathfrak{A}$ be the hom-Gerstenhaber algebra corresponding to the hom-Lie algebroid structure on $\psi^{!}TM$. Here, the map $d_{\pi}:\mathfrak{A}\rightarrow \mathfrak{A}$ defined by $d_{\pi}=[\pi,-]_{\mathfrak{A}}$ makes $\mathfrak{A}$ a strong differential regular hom-Gerstenhaber algebra (since $\pi$ is a hom-Poisson bivector). Furthermore, the map $d_{\pi}=d_*$, differential obtained by equation \eqref{NCBD} for the hom-Lie algebroid structure on $\psi^{!}T^*M$. Therefore, the hom-Lie bialgebroid structure on $(\psi^{!}TM,\psi^{!}T^*M)$ corresponding to this strong differential regular hom-Gerstenhaber algebra is the canonical hom-Lie bialgebroid structure on $(\psi^{!}TM,\psi^{!}T^*M)$ for a given hom-Poisson manifold as defined in \cite{hom-Lie1}.
\end{Exm}

\begin{Exm}
Let $(\mathfrak{g},[-,-]_{\mathfrak{g}},\alpha)$ and $(\mathfrak{g}^*,[-,-]_{\mathfrak{g}^*},\alpha^{\dagger})$ be two regular hom-Lie algebras (where $\alpha^{\dagger}(\xi)(X)=\xi(\alpha^{-1}(X))$ for $\xi\in L^*$ and $X\in L$). If $(L,L^*)$ is a purely hom-Lie bialgebra, then $(L,L^*)$ is a hom-Lie bialgebroid over a point manifold. In this case the associated strong differential hom-Gerstenhaber algebra is $(\mathfrak{G}=\wedge^*\mathfrak{g},\wedge,[-,-]_{\mathfrak{G}},\alpha_{\mathfrak{G}},d_*)$ given in Example \ref{sdga1}.
\end{Exm}

\vspace{.25cm}
{\bf Ashis Mandal and  Satyendra Kumar Mishra}\\
 Department of Mathematics and Statistics,
Indian Institute of Technology,
Kanpur 208016, India.\\
e-mail: amandal@iitk.ac.in, ~~ satyendm@iitk.ac.in

\end{document}